%% file: SRPAT_final.tex
\documentclass[letterpaper,11pt,reqno]{amsart} 

\usepackage[portrait,margin=1in]{geometry}
\usepackage{comment}

\input{style.tex}
\usepackage{tikz}
\usepackage{pgfplots}
\pgfplotsset{compat=1.16}
\newcommand{\invisible}[1]{}

\newcommand{\ee}{\mathrm{e}} 
\newcommand{\citesno}[1]{\cite{#1}}
\newcommand{\ch}[1]{{\color{red}{#1}\color{black}}}

\begin{document}

\title[Network evolution with self-reinforcement]{Network evolution with self-reinforcement}

\author[Bhamidi]{Shankar Bhamidi$^1$}
\address{$^1$Department of Statistics and Operations Research, 304 Hanes Hall, University of North Carolina, Chapel Hill, NC 27599}
\email{bhamidi@email.unc.edu}

\author[van der Hofstad]{Remco van der Hofstad$^2$}
\address{$^2$Department of Mathematics and Computer Science,  Eindhoven University of Technology, Eindhoven, The Netherlands}
\email{rhofstad@win.tue.nl}

\author[den Hollander]{Frank den Hollander$^3$}
\address{$^3$Mathematical Institute, Leiden University, Einsteinweg 55, 2333 CC Leiden, The Netherlands}
\email{denholla@math.leidenuniv.nl}

\author[Ray]{Rounak Ray$^4$}
\address{$^4$Division of Applied Mathematics, Brown University, 182 George Street, RI 02912}
\email{rounak\_ray@brown.edu}

\date{\today}

\subjclass[2010]{Primary: 05C80, 60C05, 60J80, 60K35}
\keywords{Dynamic networks, preferential attachment, self-reinforcement, growing random trees, local convergence, continuous-time branching processes}

\begin{abstract}
We study a new class of preferential attachment trees with \emph{self-reinforcement}. At each time, each vertex is assigned a weight equal to the cumulative sum over past times of an affine function of its degree. A new vertex attaches itself via a single edge to an already present vertex with a probability proportional to the current weight of that vertex. This ``integrated popularity'' rule builds long memory directly into the attachment mechanism, thereby destroying the Markov and partial-exchangeability features that underlie the classical analysis of affine preferential attachment models. More broadly, the model connects to applied-probability work on long-memory self-interacting processes (such as the elephant random walk), emphasizing how non-Markovian reinforcement reshapes asymptotic behaviour.

Despite this loss of structure, we identify an explicit exponent $\phi=\phi(\delta)$ governing both local and global growth: typical degrees at time $n$ scale as $n^{1/\phi}$, and the empirical degree distribution converges to a power-law with a tail exponent $\phi+1$. We further prove Benjamini--Schramm local convergence to an infinite random rooted tree characterized via an embedded continuous-time branching process. The limiting tree is a \texttt{sin}-tree, and is \emph{not} the P\'olya-type limiting tree arising in the non-reinforced setting. Our results provide a tractable probabilistic description of a natural ``memoryful'' network-growth mechanism, and quantify precisely how reinforcement renormalizes the classical preferential-attachment exponents.
\end{abstract}

\maketitle


\section{Introduction}
\label{sec:intro}

The last few years have seen the formulation of a host of new mathematical models capturing macroscopic properties of real-world networks (see e.g.~\citesno{albert2002statistical,newman2003structure,newman2010networks,bollobas2001random,durrett-rg-book,van2017random} and the references therein). Many networks are not static objects, but evolve on the same time scale as the dynamics they support. This observation has motivated a large literature on \emph{temporal} (or dynamic) networks, where one tracks both the graph topology and the time-ordering of interactions (see e.g.~\citesno{holme2012temporal,masuda2016guidance} and the references therein). A common modeling theme is that microscopic rules for the arrival of vertices/edges and their attachment choices can produce macroscopic signatures such as heavy-tailed degree distributions, aging effects, and strong correlations between vertex age and connectivity.

Preferential attachment is a paradigmatic example: it implements ``cumulative advantage'' by biasing attachment toward already well-connected vertices, and yields explicit scaling exponents and local limits in different variants. However, in applications the attraction of a vertex is often not well represented by its \emph{current} degree alone, and may depend on its \emph{interaction history} (visibility accumulated over time, repeated exposure, persistent reputation, or reinforcement through recommendation loops). This motivates growth mechanisms with genuine \emph{memory}. From a probabilistic viewpoint, such mechanisms naturally connect to the broad theory of reinforced stochastic processes (P\'olya urns, reinforced walks, self-interacting processes), where history-dependent choice rules can change scaling and limit objects (see~\cite{pemantle2007survey} for an overview).

In this paper we study a growing tree model in which attachment is biased not by instantaneous degree, but by an \emph{integrated} degree signal: the weight of a vertex at time $n$ is the cumulative sum, over all past times, of an affine function of its degree. Equivalently, each time a vertex gains a new neighbor, this event permanently increases its future attractiveness, and the effect compounds over subsequent times through repeated inclusion in the running sum. This ``integrated popularity'' mechanism is natural in settings where attention is aggregated over long horizons: a webpage's attractiveness can scale with cumulative traffic; a user's visibility can depend on accumulated impressions rather than current follower count; and algorithmic ranking system can reinforce items by using time-averaged engagement rather than purely instantaneous metrics.

The self-reinforced rule should be contrasted with two classical ways to enrich preferential attachment. One direction changes the attachment \emph{function} (sublinear, affine, fitness, aging), but still uses the current network state as sufficient information. Another direction, closer in spirit to the ``information constraint'' viewpoint, assumes that incoming vertices act on \emph{outdated} or \emph{partial} information about the network, producing nontrivial macroscopic effects even when the local rule remains preferential attachment. For example, in the delay framework of~\citesno{BBDS04_meso,BBDS04_macro}, a new vertex attaches by using the network observed at a delayed time, yielding a different class of limits driven by the interplay between growth and information latency. Our model takes an orthogonal route: vertices see the present network, but the attachment propensity itself carries the entire past through a cumulative functional of the degree process. This is precisely what breaks the standard time-inhomogeneous Markov structure for a fixed vertex degree, and invalidates the usual partial-exchangeability or P\'olya-urn representations available for non-reinforced affine preferential attachment. 

We next introduce the model studied in this paper, informally describe the main contributions of the paper, and outline the remainder of the paper.


\subsection{Network evolution with self-reinforcement}
\label{sec:net-mod}

Start at time $n=0$ with an isolated vertex $v_0$. At time $n=1$, a vertex $v_1$ comes in and attaches itself to vertex $v_0$ via a single edge. At subsequent times the growth is as follows. For $n \geq 1$, let 
\begin{equation}
\theta(v_i,n) =  \text{ weight of vertex } v_i \text{ at time } n, \qquad  0 \leq i \leq n. 
\end{equation}
At time $n \geq 1$, a vertex $v_{n+1}$ enters and attaches itself via a single edge to one of the vertices already present in the graph with a probability that is proportional to the \emph{current weight} of that vertex. 

The specific choice of weights considered in this paper is the following. For $n \geq 1$, let 
\[
d(v_i,n) = \text{ degree of vertex } v_i \text{ at time } n, \qquad  0 \leq i \leq n.
\] 
Note that $d(v_0,1) = d(v_1,1) = 1$, while $d(v_i,j) = 0$ for $1 \leq j < i$ and $d(v_i,i) = 1$ for all $i \geq 2$. Let $\{n+1 \sim i\}$ denote the event that vertex $v_{n+1}$ attaches itself to vertex $v_i$. For every $n \geq 1$, the probability of this event equals
\[
\frac{\theta(v_i,n)}{\sum_{j=0}^n \theta(v_j,n)}, \qquad 0 \leq i \leq n,
\]
where the weight is chosen to be 
\begin{equation}
\label{eq:weightchoice}
\theta(v_i,n) = \sum_{m=i}^{n} \big(d(v_i,m) + \delta\big), \qquad 0 \leq i \leq n,
\end{equation}
with $\delta > -1$ a parameter. Because at time $m$ there are $m+1$ vertices and $m$ edges, we have
\begin{equation}
\sum_{j=0}^n \theta(v_j,n) = \sum_{j=0}^n \sum_{m=j}^{n} \big(d(v_j,m) + \delta\big) 
= \sum_{m=0}^n 2m + \delta \sum_{m=0}^n m = n(n+1)(1+\tfrac12\delta).
\end{equation}

In what follows, $(\cT_n)_{n\geq 1}$ denotes the growing sequence of random trees where, by convention, $\cT_0 = \cT_1$), and $\cL$ denotes its probability law. Let $\{\FF_n\}_{n \geq 1}$ be the filtration defined by 
\[
\FF_n = \sigma\big(d(v_i,m)\colon\,1 \leq m \leq n,\, 0 \leq i \leq n\big).
\]


\subsection{Main contributions of this paper}

Our main results provide a sharp and explicit scaling theory for the growth mechanism with memory. We identify a deterministic exponent $\phi=\phi(\delta)$ that simultaneously governs: (i) the polynomial growth scale of typical degrees; (ii) the tail exponent of the limiting empirical degree distribution. We establish (Benjamini--Schramm) local convergence to an infinite random rooted sin-tree, which admits a characterization in terms of an embedded continuous-time branching process. The limiting local object differs from the P\'olya-type local limits in the non-reinforced model, reflecting that reinforcement changes the effective reproduction mechanism along typical ancestral lines.

From a technical perspective, the key challenge is that the attachment weights are \emph{path functionals} of the degree evolution. This memory prevents the degree of a fixed vertex being treated as a Markov chain and breaks the classical continuous-time branching-process embedding used for standard preferential attachment models. Our approach overcomes this by: (a) identifying the correct time scale induced by the integrated weights; (b) developing a coupling/embedding that restores an {\em approximate} branching structure at the local level, allowing us to read off both local limits and degree exponents.


\subsection{Outline of the remainder of the paper}

Section~\ref{sec:main-res} introduces the continuum limit object via a memory-driven branching process (Section~\ref{sec:lim-obj-macro}), states the main local and degree results (Section~\ref{sec:mainresults}), provides intuition (Section~\ref{intuition}), and offers a discussion (Section~\ref{sec:discussion}). Section~\ref{sec:local-wll} recalls the local-convergence framework for rooted trees and \texttt{sin}-trees, including the fringe decomposition tools used later. Section~\ref{edgedesc} gives an equivalent edge-based construction of the limiting branching process together with moment bounds needed in the proofs. Finally, Section~\ref{sec:proofs} contains the proofs: local convergence in Section~\ref{sec:proof-main-local}, tail behavior in Section~\ref{sec:tail-exp}, and root-degree growth in Section~\ref{sec:max-deg-proof}.


\section{Results}
\label{sec:main-res}

We need some notation to describe the limit objects.


\subsection{A branching process driven by point processes with memory}
\label{sec:lim-obj-macro}

Recall that we consider the \emph{self-reinforcement} as an affine preferential attachment class with attachment function $f(k) = k+\delta$ for $k\geq 1$, for a fixed parameter $\delta>-1$. In the construction of continuous-time branching processes~\citesno{jagers-ctbp-book,jagers-nerman-1,jagers-nerman-2}, one formulation is via a point process $\xi$ with inter-arrival times $(\sE_{(i-1) \leadsto i})_{i\geq 1}$, where, conceptually, for each $i\geq 1$, $\sE_{(i-1) \leadsto i}$ (in the branching process) has the interpretation as the amount of time for a vertex to go from having $i-1$ to $i$ children. We will specify a point process via a recursive construction of its inter-arrival times. The construction proceeds as follows: 
\begin{enumeratea}
\item 
\emph{Base case $\sE_{0 \leadsto 1}$:}  Define the hazard function
\begin{align}
\label{eqn:h01-hazard}
h_{0\leadsto 1}(x) = \frac{1+\delta}{1+\tfrac12\delta}\,(1-\eee^{-x}), \qquad x \geq 0.
\end{align}
Let $\sE_{0\leadsto 1}$ be the random variable on $[0,\infty]$ with the above hazard rate, so that, for any $x$, $\pr(\sE_{0\leadsto 1} > s) = \exp(-\int_0^s h_{0\leadsto 1}(x)\, \dd x)$.  
\item 
\emph{General case (general $i$):} 
Given $(\sE_{(j-1) \leadsto j})_{1\leq j \leq k-1}$, define the hazard function
\begin{align}
\label{eqn:h-gen-hazard}
h_{(k-1)\leadsto k}(x) = \frac{1+\delta}{1+\frac{1}{2}\delta}(1- \eee^{-(x+\sigma_{k-1})}) + \frac{1}{1+\frac{1}{2}\delta}\sum_{j=1}^{k-1}(1- \eee^{- (\sigma_{k-1} - \sigma_j +x)}), \qquad x \geq 0,
\end{align}
where $\sigma_j = \sum_{i=1}^j \sE_{i-1 \leadsto i}$. Let $\sE_{(k-1) \leadsto k} >0$ a.s.\ be the random variable with the above hazard rate, so that, for any $x$, $\pr(\sE_{(k-1) \leadsto k} > s\mid (\sigma_j)_{j=1}^{k-1}) = \exp(-\int_0^s h_{(k-1)\leadsto k}(x)\, \dd x)$.
\end{enumeratea}
 
Let $\cP_{\Sr, \delta}$ be the above birth process. As shown in Theorem~\ref{thm:main-local}, the local limit of our network model is given by a continuous-time branching process, with reproductions driven by $\cP_{\Sr, \delta}$, stopped at an independent exponential time. We next define these objects.
   
\begin{defn}[BP in the macroscopic regime]
\label{def:limit-bp-macro}
Let $\BP_\Sr(\cdot)$ be the continuous-time branching process, starting from $1$ individual at time $t=0$, where each individual has offspring distribution $\cP_\Sr$ defined above. Let $T_{1}$ be an $\exp(1)$ random variable, independent of $\BP_\Sr$, and write $\varpi_{\Sr}$ for the distribution of $\BP_\Sr(T_1)$, viewed as a random finite rooted tree, where only the genealogical information between individuals in $\BP_{\Sr, \delta}(T_1)$ is retained. 
\end{defn}


\subsection{Main results}
\label{sec:mainresults}

We start by describing the local limit of our self-reinforced tree. Recall from~\eqref{eqn:deg-count} that $N_{k}(n)$ denotes the number of vertices in $\cT_n$ with degree exactly $k$, for $k\geq 1$. 

\begin{thm}[Local limit, macroscopic regime]
\label{thm:main-local}The sequence of random trees $(\cT_n)_{n\geq 1}$ with law $\cL(\delta)$ converges in probability in the local convergence sense (see Definition~\ref{def:local-weak} and \eqref{it:fringe-exp}) to the unique {\tt sin}-tree with distribution $\varpi_{\Sr,\delta}$ (see Definition~\ref{def:limit-bp-macro}). In particular, the degree distribution satisfies 
\[
 \frac{N_k(n)}{n} \probc \pr\left(\sum_{j=1}^{k-1} \sE_{(j-1)\leadsto j} 
\leq T_1 < \sum_{j=1}^{k} \sE_{(j-1)\leadsto j}\right) = p_{\Sr, \delta}(k), \qquad k\geq 1,
 \]
where $T_1$ is a rate-$1$ exponential random variable independent of $(\sE_{(j-1) \leadsto j})_{j\geq 1}$. 
\end{thm}

\begin{figure}[htbp]
\centering
\begin{minipage}{0.48\textwidth}
\centering
\includegraphics[width=\linewidth]{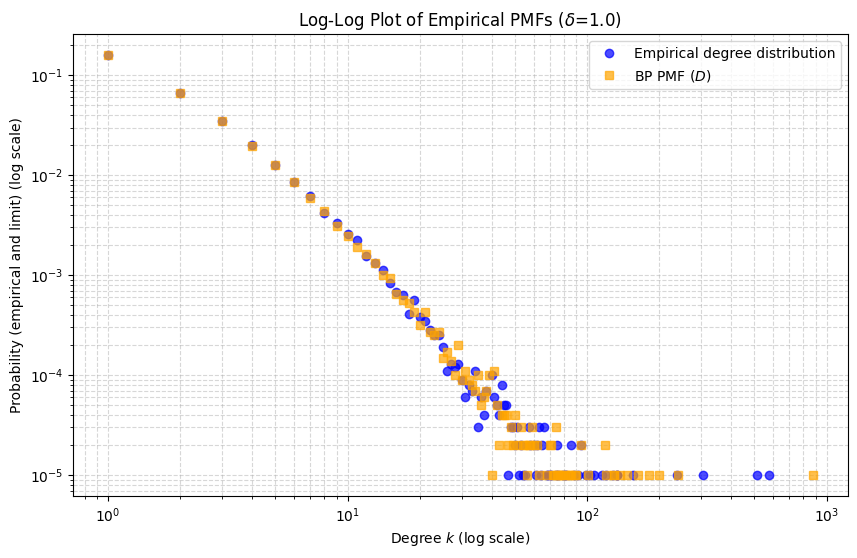}
\end{minipage}
\hfill
\begin{minipage}{0.48\textwidth}
\centering
\includegraphics[width=\linewidth]{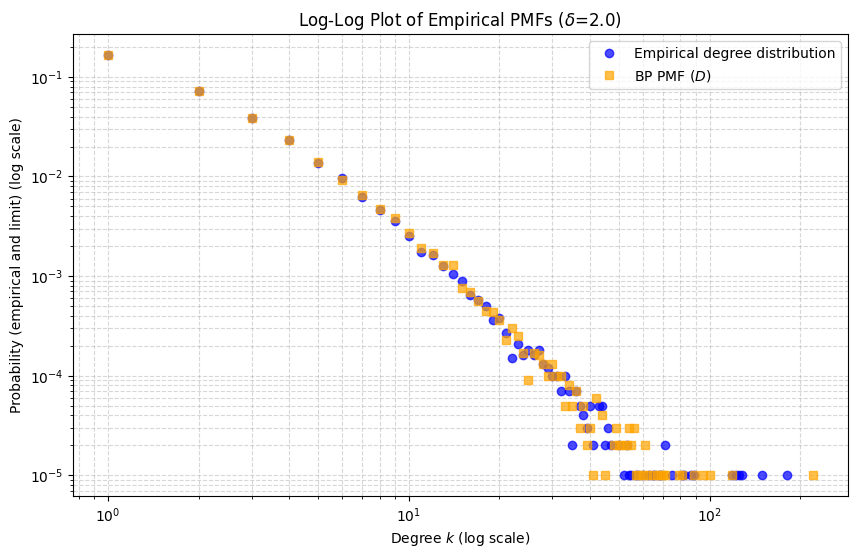}
\end{minipage}
\caption{Empirical degree distribution and limit probability mass function for $\delta=1$ and $\delta=2$.}
\label{fig:delta_pair}
\end{figure}

\begin{figure}[htbp]
\centering
\begin{minipage}{0.48\textwidth}
\centering
\includegraphics[width=\linewidth]{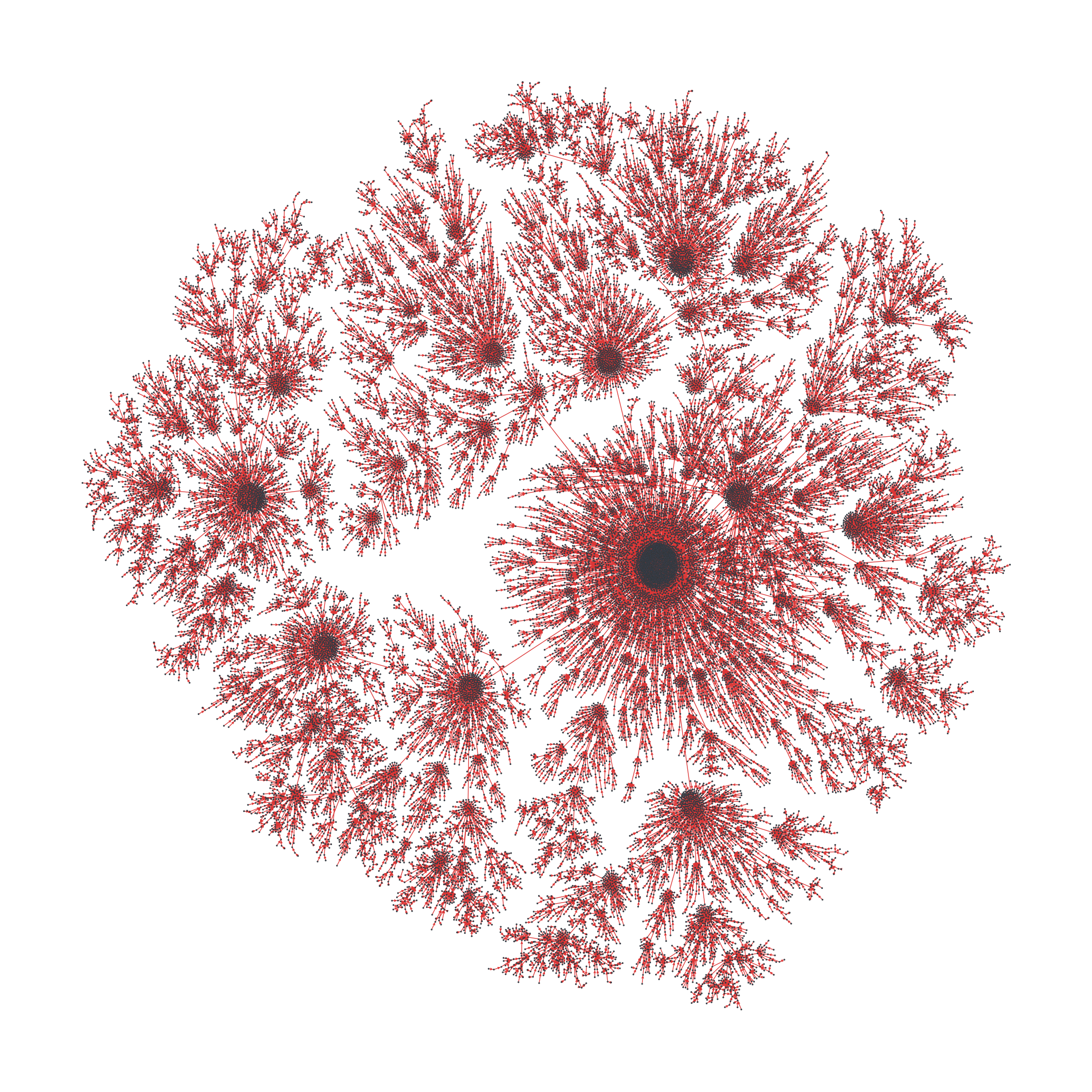}
\end{minipage}
\hfill
\begin{minipage}{0.48\textwidth}
\centering
\includegraphics[width=\linewidth]{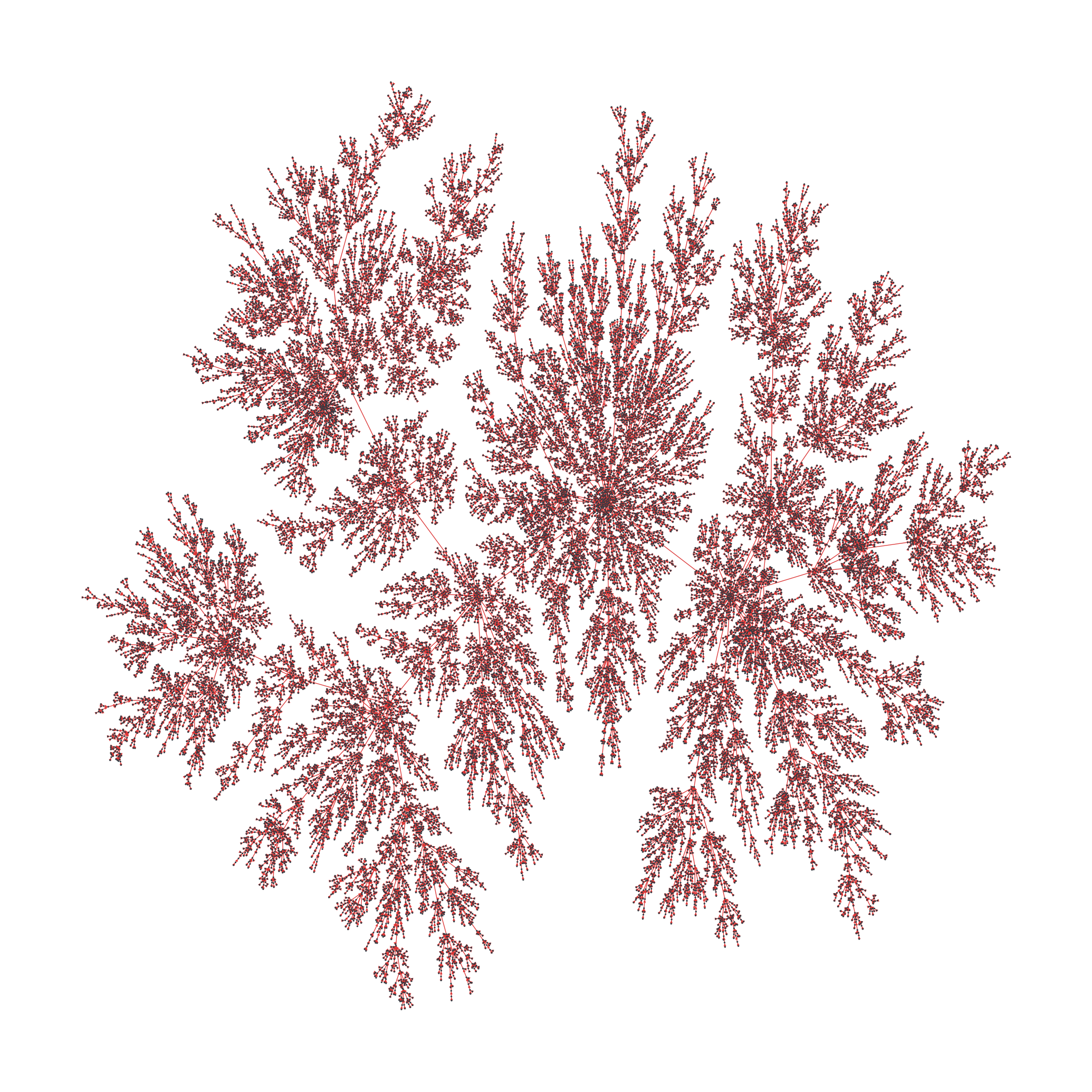}
\end{minipage}
\caption{\small Two realizations of preferential attachment trees with self-reinforcement grown to the same size ($n=40,000$), illustrating the effect of the parameter $\delta$ on network geometry. The left panel corresponds to $\delta=0$ and exhibits a markedly more hub-dominated structure, with attachment concentrating more strongly around a few high-degree vertices. The right panel corresponds to $\delta=10$, where attachment is more diffuse and the resulting tree appears more evenly spread, with less pronounced hubs.}
\label{fig:delta_pair_40000}
\end{figure}

The above result immediately implies the following. Let $\vA_n$ denote the adjacency matrix of $\cT_n$. Let $(\lambda_i^{\sss(n)})_{1\leq i\leq n}$ denote the set of eigenvalues of $\vA_n$, and $\hat{\mu}_n = n^{-1} \sum_{i=1}^n \delta_{\lambda_i^{\sss(n)}}$ the empirical spectral distribution of $\vA_n$ (where $\delta_{(\cdot)}$ denotes the Dirac delta-function). 

\begin{cor}[Limiting empirical spectral distribution]
\label{cor:rand-adj}
There exists a deterministic distribution $\mu_\infty$ (whose specific form depends on the parameter $\delta$) such that the empirical spectral distribution satisfies $\hat{\mu}_n \convd \mu_\infty$. The limit distribution has an infinite set of atoms in $\bR$.
\end{cor}

\noindent
The proof of this corollary follows by combining the extended fringe convergence in Theorem~\ref{thm:main-local} with~\cite{bhamidi2012spectra}*{Theorem 4.1}. 

The goal of the next few results is to understand the tail behavior of the degree distribution. Without self-reinforcement, the degree exponent, i.e., the power-law rate of decay of the tail probabilities of the limiting degree distribution, is $2+\delta$~\citesno{barabasi1999emergence,bollobas2001degree}. The edge branching process construction described in Section~\ref{edgedesc} will play a major role, both in the proof of Theorem~\ref{thm:tail} and in the analysis of the degree exponents below. 

The next theorem describes tail asymptotics for the limit degree distribution. We will use $p_{\Sr, \delta}(\geq k) = \sum_{l=k}^\infty p_{\Sr, \delta}(l)$ for the tail of the degree distribution. Define
\begin{equation}
\label{phidef}
\phi = \phi(\delta) = \tfrac12(1+\tfrac12\delta) \left[1 + \sqrt{1+4(1+\tfrac12\delta)^{-1}}\,\right].
\end{equation}

\begin{thm}[Tail exponents]
\label{thm:tail}
There exist finite positive constants $C_1, C_2$ such that 
\[
C_{2} k^{-\phi(\delta)} \leq p_{\Sr, \delta}(\geq k) \leq C_{2} k^{-\phi(\delta)}, \qquad k \geq 1.
\]
\end{thm}

Our third theorem identifies the growth of the degrees. 

\begin{thm}[Degree growth]
\label{degree_growth} 
For every $i\geq 1$, there exist an $L^1$ random variables $\zeta_i $ with $\E(\zeta_i)>0$ such that $n^{-1/\phi(\delta)} d(v_i,n) \to  \zeta_i$ almost surely and in $L^1$.
\end{thm}


\subsection{Intuition for the limit object}
\label{intuition} 

Before formally exhibiting the local weak limit, we give an intuitive picture of why the above continuous-time branching process is expected to show up in the local asymptotics. To simplify the notation, assume that $\delta =0$ throughout. Let us start by considering a ``continuous time'' analog $\Cod$ of the network process (at this stage, all that we are doing is changing the time scale at which events happen, moving from discrete to continuous time). Start with one vertex $\Cod(0) = \set{\rho}$ at time $t=0$. Suppose that, at any given time $t$, each vertex $v\in \Cod(t)$ begets a new vertex at rate $1$. This new vertex is not necessarily a child of $v$. Rather, it will connect to a vertex in $\Cod(t)$ by using the original discrete-time dynamics. 

The dynamics of $\Cod$ imply that the size of the process $(|\Cod(t)|)_{t\geq 0}$ has the same distribution as a rate-$1$ Yule process. Standard results about the Yule process imply that $\eee^{-t}|\Cod(t)| \stackrel{a.s.,~L^2}{\to} W$ with $W\sim \exp(1)$. Let $V_t$ denote a randomly selected vertex in $\Cod(t)$, born at time $B(V_t)$. Then its \emph{age} by time $t$, $\age(V_t) = t- B(V_t)$, satisfies 
\begin{equation}
\label{eqn:958}
\pr(\age(V_t)> a\mid \Cod(t)) = \frac{|\Cod(t-a)|}{|\Cod(t)|}\to \eee^{-a},
\end{equation}
so that the age has approximate distribution
\begin{equation}
\age(V_t) \approx \exp(1). 
\end{equation}
This explains the origin of the random variable $T_1$ in Definition~\ref{def:limit-bp-macro}, describing the limit of the fringe distribution as $\BP_{\Sr}$ run for an exponential rate-$1$ amount of time. 

Next, fix a large time $t$ and suppose that a vertex $v$ was born at time $t$. We explain heuristically how we can get a handle on the evolution of its degree, starting from the transition from degree $1$ (leaf) to degree $2$ via a new vertex attaching to it. Fix $x> 0$, and consider the hazard rate of a new vertex connecting to vertex $v$ at time $t+x$, namely, conditional on no connections to $v$ until this time, consider the probability of a new connection forming in the time interval $[t+x, t+x+\Delta x]$. The rate at which a new individual is born in this time interval is $|\BP_{\Sr}(t+x)|\, \Delta x \approx W\exp(t+x)$. Using the (discrete-time) probabilistic mechanism of edge formation, we see that this infinitesimal probability is approximately given by 
\begin{align*}
|\BP_{\Sr}(t+x)|\,\Delta x\, \frac{(|\BP_{\Sr}(t+x)| - |\BP_{\Sr}(t)|)}{\BP(t+x) (|\BP_{\Sr}(t+x)|+1)} 
\approx \frac{W\eee^{t+x} - W\eee^t}{W\eee^{t+x}} \Delta xv\approx (1-\eee^{-x})\,\Delta x,
\end{align*}
which explains the hazard rate in \eqref{eqn:h01-hazard}. This argument can then be extended in a straightforward manner to understand the hazard rate for the time to form additional edges. 


\subsection{Discussion}
\label{sec:discussion}

In the standard \emph{preferential attachment model}, the weight is given by the last term of the sum in \eqref{eq:weightchoice}, i.e., $\theta(v_i,n) = d(v_i,n) + \delta$. This model has been investigated in detail and exhibits a scaling behavior that is similar to that in Theorems~\ref{thm:main-local}--\ref{degree_growth}, but simpler (see \citesno{van2017random,van2024random} and references therein). When the weight includes the sum over all priori times, as in \eqref{eq:weightchoice}, we refer to the process as the \emph{self-reinforced preferential attachment model}. The reason for this name is that the latter model has a flavor similar to that of edge-self-reinforced random walk, where edge-crossing probabilities are proportional to the total number of edge-crossings at prior times. A closely related source of ``complete-history'' memory in probability and statistical physics is the \emph{elephant random walk}, where the next increment is chosen by sampling uniformly from the entire past and copying (or flipping) a previously taken step \citesno{schutz2004elephant,dasilva2013nongaussian,baur2016elephant,bercu2018martingale,laulin2022elephant,bertenghi2024thesis}. While the state spaces and observables differ (positions versus degrees), both models exhibit non-Markovian dynamics driven by a reinforcement mechanism that aggregates information over all earlier times. The special case where $\delta=0$ was introduced and analysed in \cite{YDFdHpr}, where the evolution of the degree of a fixed vertex was analyzed. 

What makes the self-reinforcement challenging is that it introduces \emph{memory} into the evolution, which is exemplified by the fact that the degree of a given vertex no longer is a time-inhomogeneous Markov chain. In turn, this implies that the growing random tree can no longer be viewed as a continuous-time branching process stopped at the moment when it reaches a given size. The presence of memory also comes with a \emph{loss of exchangeability}, so that, for instance, the finite graph P\'olya urn representation \citesno{BBCS05,berger2014asymptotic} is no longer valid. All these hurdles have to be overcome in the proof of Theorems~\ref{thm:main-local}--\ref{degree_growth}.     

It is easily checked that $\phi(\delta) < 2+\delta$ for all $\delta > -1$, where the upper bound is the counterpart of $\phi(\delta)$ in the model without self-reinforcement. Thus, as expected, the degrees grow faster and the tail of the degree distribution is thicker with self-reinforcement. Moreover, $\delta \mapsto \phi(\delta)$ is strictly increasing and strictly concave, with $\lim_{\delta \downarrow -1} \phi(\delta) =1$ and $\lim_{\delta \to \infty} \phi(\delta)/\delta = \tfrac12$. Thus, as expected, increasing $\delta$ hampers the formation of large degrees because only few of the vertices attach to the rare hubs. See Figure~\ref{fig:map} for a plot.

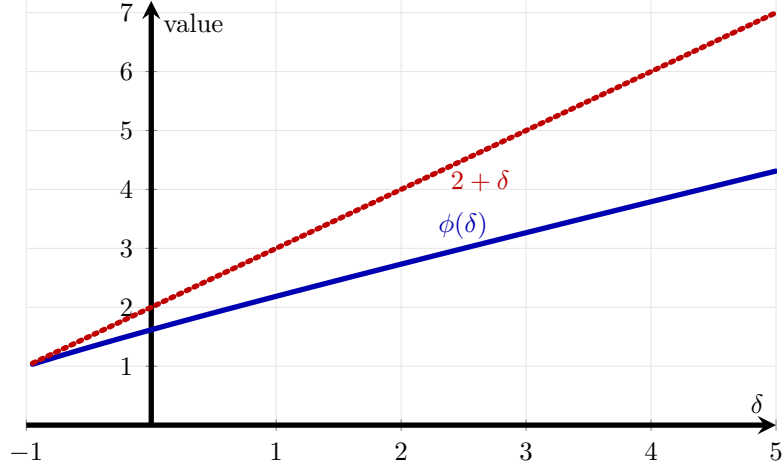
\begin{figure}[htbp]
\centering
\begin{tikzpicture}
\begin{axis}[
    width=11.5cm,
    height=7.2cm,
    domain=-0.95:5,
    samples=300,
    axis lines=middle,
    xlabel={$\delta$},
    ylabel={value},
    xmin=-1, xmax=5,
    ymin=0, ymax=7.2,
    xtick={-1,0,1,2,3,4,5},
    ytick={1,2,3,4,5,6,7},
    grid=major,
    grid style={gray!20},
    tick label style={font=\small},
    label style={font=\small},
    every axis plot/.append style={line cap=round}
]

\addplot[
    very thick,
    blue!70!black
]
{0.5*(1+0.5*x)*(1 + sqrt(1 + 4/(1+0.5*x)))};

\addplot[
    very thick,
    red!75!black,
    dotted
]
{x+2};

\node[blue!70!black, font=\small, anchor=south west]
    at (axis cs:2.2,2.95) {$\phi(\delta)$};

\node[red!75!black, font=\small, anchor=west]
    at (axis cs:2.3,4.15) {$2+\delta$};

\end{axis}
\end{tikzpicture}
\caption{\small The blue drawn curve is a plot of the function $\delta \mapsto \phi(\delta)$. The red dotted curve is its counterpart $\delta \mapsto 2+\delta$ for standard affine preferential attachment (in the absence of self-reinforcement). Note that the asymptotic slopes of the two curves are different, and that the lower curve is close to linear but is in fact strictly concave. In particular, $\phi(\delta) = \tfrac12[\delta + 4 + O(1/\delta)]$ as $\delta\to\infty$.}
\label{fig:map}
\end{figure}


\section{Local convergence and {\tt sin}-trees}
\label{sec:local-wll}

The main goal of the present paper is to prove that the sequence of random networks $\{\cT_n\colon\,n \geq 1\}$ converges in an appropriate sense to a limiting infinite tree. The goal of this section is to describe the by now standard notion of \emph{local convergence}~\citesno{aldous-steele-obj,benjamini-schramm,van2024random}. We follow Aldous in~\cite{aldous-fringe}, who developed the foundations of local limits for probabilistic models of trees, where local convergence has an equivalent formulation in terms of convergence of the fringe distribution for the sequence of random trees corresponding to infinite trees with a single infinite path to infinity (sometimes referred to as {\tt sin}-trees). The work in~\cite{aldous-fringe} in the more involved setting of trees with vertex attributes was explained in~\cite{antunes2023attribute}, which we paraphrase in our setting. 


\subsection{Mathematical notation}

Write $\stod$ for stochastic domination between two real-valued probability measures. For $J\geq 1$, let $[J] = \set{1,\ldots, J}$. If $Y$ has an exponential distribution with rate $\gl$, then write $Y \sim \exp(\gl)$. Write $\convas,\convp,\convd$ for convergence almost surely, in probability and in distribution, respectively. 

A sequence of events $(A_n)_{n\geq 1}$ is said to occur \emph{with high probability} (whp) when $\pr(A_n) \to 1$ as $n \to \infty$. One of the core objects in the present paper is the sequence of growing random trees $(\cT_n)_{n\geq 1}$. Throughout, $d(v,\cT_n)$ denotes the degree of the vertex $v$ in the tree $\cT_n$, and
\begin{equation}
    \label{eqn:deg-count}
    N_k(n) = \sum_{v\in \cT_n} \ind\set{d(v,\cT_n) = k}, \qquad k\geq 1,
\end{equation}
denotes the empirical degree counts at time $n$.

  
\subsection{Fringe decomposition for trees}

For $n\geq 1$, let $\bT_{n}$ be the space of all rooted trees on $n$ vertices. Let $ \bbT = \cup_{n=0}^\infty \bT_{n} $ be the space of all finite rooted trees. Here, $\bT_{0} = \emptyset$ will be used to represent the empty tree (no vertices). Let $\rho_{\bt}$ denote the root of $\bt$. For $r\geq 0$ and $\bt\in \bbT$, let $B(\bt, r) \in \bbT$ denote the subgraph of $\bt$ consisting of those vertices within graph distance $r$ from $\rho_{\bt}$, viewed as an element of $\bbT$ and rooted again at $\rho_{\bt}$. 

Given two rooted finite trees $\bs, \bt \in \bbT$, we write $\bs \simeq \bt$ when there exists a \emph{root preserving} isomorphism between the two trees viewed as unlabelled graphs. Given two rooted trees $\bt,\bs \in \bbT$ (\cite{benjamini-schramm},~\cite{van2024random}*{(2.3.15)}), define the distance 
\begin{align}
\label{eqn:distance-trees}
	d_{\bbT}(\bt,\bs) = \frac{1}{1+R^*}, \qquad R^* =\sup\{r\colon\, B(\bt, r) \simeq B(\bs, r)\}.
\end{align}

Next, fix a tree $\bt\in \bbT$ with root $\rho = \rho_\bt$ and a vertex $v\in \bt$ at graph distance $h$ from the root. Let $(v_0 = v, v_1, \ldots, v_h = \rho)$ be the unique path from $v$ to $\rho$. The tree $\bt$ can be decomposed into $h+1$ rooted trees $f_0(v,\bt), \ldots, f_h(v,\bt)$, where $f_0(v,\bt)$ is the tree rooted at $v$ consisting of all the vertices for which there exists a path from the root passing through $v$. For $i \ge 1$, $f_i(v,\bt)$ is the subtree rooted at $v_i$, consisting of all the vertices for which the path from the root passes through $v_i$ but not through $v_{i-1}$ (see Figure~\ref{fig:fringe}). 

\begin{figure}[htbp]
\centering
\includegraphics[scale=.2]{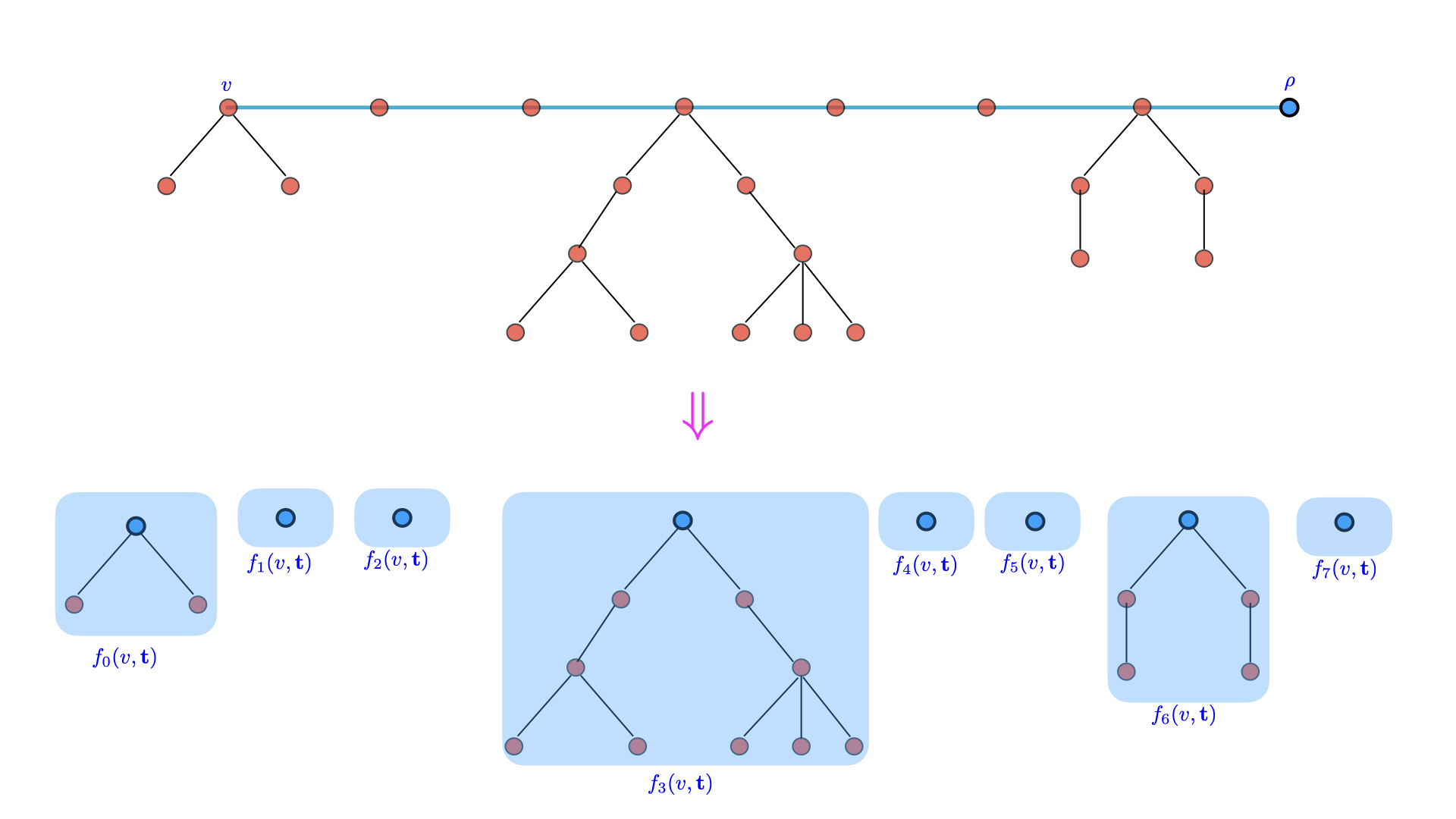}
\caption{Fringe decomposition around vertex $v$ of a finite tree rooted at $\rho$. The blue colored vertices represent the roots of the respective trees.}
\label{fig:fringe}
\end{figure}
 
Call the map $(v,\bt) \leadsto \bbT^\infty$ (i.e., infinite sequences of finite trees) with $v\in \bt$, defined by 
\[
F(v, \bt) = \left(f_0(v,\bt), f_1(v,\bt) , \ldots, f_h(v,\bt), \emptyset, \emptyset, \ldots \right),
\]
the \emph{fringe decomposition} of $\bt$ about the vertex $v$. Call $f_0(v,\bt)$ the \emph{fringe} of the tree $\bt$ at $v$. For $k\geq 0$, call $F_k(v,\bt) = (f_0(v,\bt), \ldots, f_k(v,\bt))$ the \emph{extended fringe} of the tree $\bt$ at $v$ truncated at distance $k$ from $v$ on the path to the root.

Next, consider the space $\bbT^\infty$. The metric in~\eqref{eqn:distance-trees} extends to $\bbT^\infty$, e.g.\ via the distance
\begin{align}
\label{eqn:dist-inf}
	d_{\bbT^\infty}((\bt_0, \bt_1, \ldots),(\bs_0, \bs_1, \ldots)) = \sum_{i=0}^\infty \frac{1}{2^i} d_{\bbT}(\bt_i, \bs_i). 
\end{align}
We can define analogous extensions to $\bT^k$ for finite $k$.  

An element $\bfomega = (\bt_0, \bt_1, \ldots) \in \bbT^\infty$, with $|\bt_i|\geq 1$ for all $ i\geq 0$, can be thought of as a locally-finite infinite rooted tree with a \emph{single} path to \emph{infinity} (called a {\tt sin}-tree~\cite{aldous-fringe}), as follows. Identify the sequence of roots of $(\bt_i)_{i\geq 0}$ with the integer lattice $\Zbold_+ = \set{0,1,2,\ldots}$, equipped with the natural nearest-neighbor edge set, rooted at $\rho=0$. Analogous to the definition of extended fringes for finite trees, for any $k\geq 0$, write $F_k(0,\bfomega)= (\bt_0, \bt_1, \ldots, \bt_k) \in \bbT^k$. See Figure~\ref{fig:sin}. 

\begin{figure}[htbp]
\centering
\includegraphics[scale=.2]{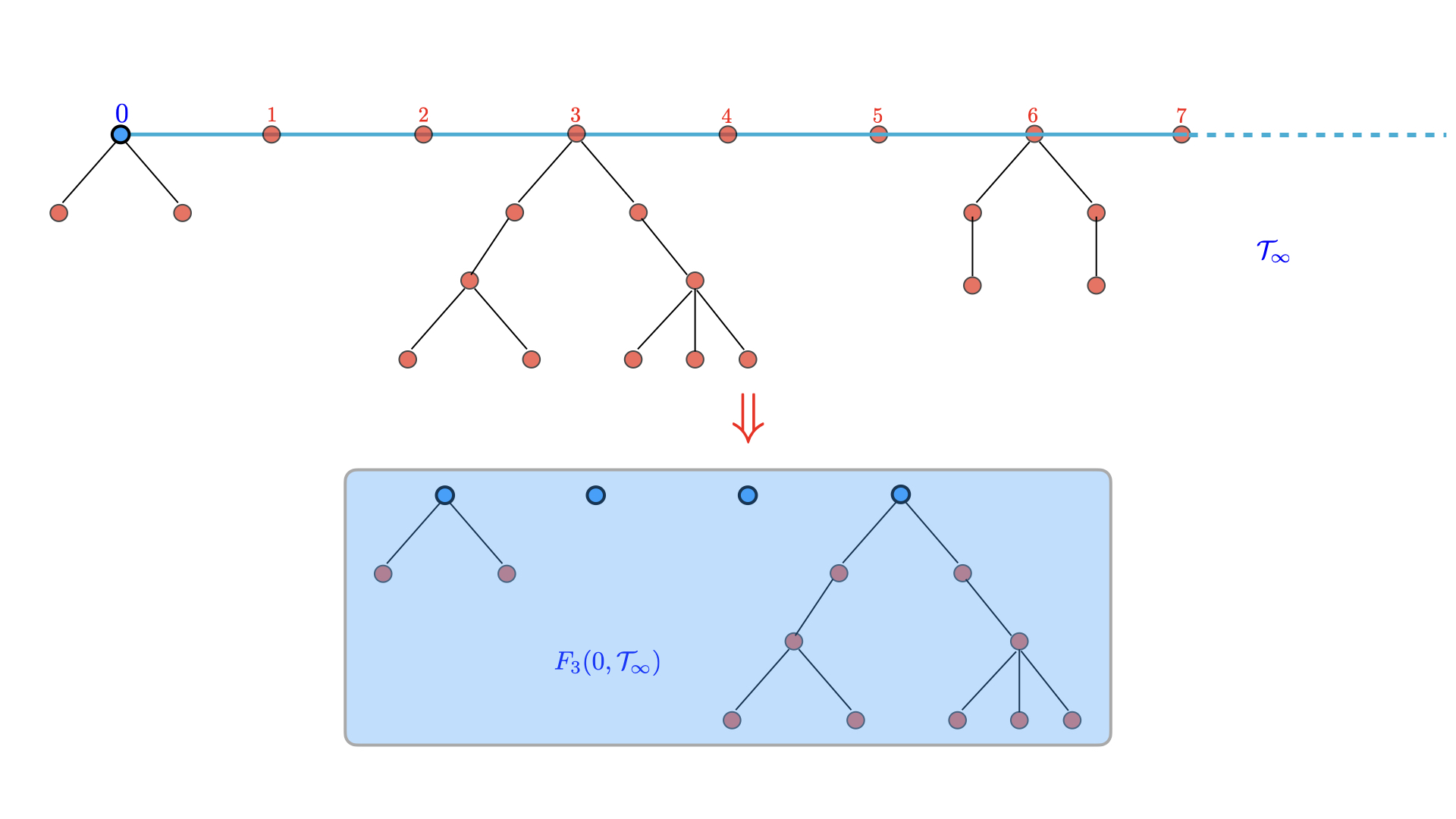}
\caption{A {\tt sin}-tree $\cT_\infty$, namely, a tree rooted at $0$ with a single infinite path to infinity, and the corresponding extended fringe $F_3(0,\cT_\infty)$ up to level 3 about $0$.}
\label{fig:sin}
\end{figure}

Call this the extended fringe of the tree $\bfomega$ at vertex $0$, until distance $k$, on the infinite path from $0$. Call $\bt_0 = F_0(0,\bfomega)$ the \emph{fringe} of the {\tt sin}-tree $\bfomega$. Suppose $\prob$ is a probability measure on $\bbT^\infty$ such that, for $\TT_\infty = (\bt_0(\TT_\infty), \bt_1(\TT_\infty),\ldots) \sim \prob$, $|\bt_i(\TT_\infty)|\geq 1$ a.s.\ for all $i\geq 0$. Then $\TT_\infty$ can be thought of as an infinite \emph{random} {\tt sin}-tree. 

Define a matrix $\vQ = (\vQ(\vs,\vt)\colon\, \vs, \vt \in \bbT)$ as follows. Suppose that the root $\rho_{\vs}$ in $\vs$ has degree $d(\rho_{\vs}, \vs) \ge 1$, and let $(v_1,\ldots, v_{d(\rho_{\vs}, \vs)})$ denote its children. For $1\leq i\leq v_{d(\rho_{\vs}, \vs)}$, let $f(\vs, v_i)$ be the subtree below $v_i$ and rooted at $v_i$, viewed as an element of $\bbT$. Define
\begin{align}
\label{eqn:Q-def}
	\vQ(\vs,\vt) = \sum_{i=1}^{d(\rho_{\vs},\vs)} \ind\set{d_{\bbT}(f(\vs, v_i), \vt) = 0}. 
\end{align} 
Thus, $\vQ(\vs, \vt)$ counts the number of descendant subtrees of the root of $\vs$ that are isomorphic (in the sense of topology) to $\vt$. For $d(\rho_{\vs}, \vs)=0$, define $\vQ(\vs, \vt)=0$. Consider a sequence $(\bar \vt_0, \bar \vt_1, \dots)$ of trees in $\bbT$ such that $\vQ(\bar \vt_i, \bar \vt_{i-1}) \ge 1$ for all $i \ge 1$. Then there exists a unique infinite {\tt sin}-tree $\TT_\infty$ with infinite path indexed by $\Zbold_+$ such that $\bar \vt_i$ is the subtree rooted at $i$ for all $i \in \Zbold_+$. Conversely, it is easy to see that, by taking $\bar \vt_i$ to be the union of (vertices and induced edges) of $\vt_0,\dots, \vt_i$ for each $i \in \Zbold_+$, every infinite {\tt sin}-tree has such a representation. Following~\cite{aldous-fringe}, we call this the \emph{monotone representation} of the {\tt sin}-tree $\TT_\infty$.


\subsubsection{Convergence on the space of trees}
\label{sec:fringe-convg-def}

For $1\leq k\leq \infty$, let $\cM_{\pr}(\bbT^k)$ denote the space of probability measures on the associated space, metrized by using the topology of weak convergence inherited from the corresponding metric on the space $\bbT^k$, resulting in it being a Polish space (see, e.g.,~\cite{billingsley2013convergence}). Suppose that $(\cT_n)_{n\geq 1} \subseteq \bbT$ is a sequence of \emph{finite} rooted random trees on some common probability space (for notational convenience, assume that $|\cT_n| = n+1,$ or, more generally, $|\cT_n|\convas \infty$). For $n\geq 1$ and $k\geq 0$, define the empirical distribution of fringes up to distance $k$ as 
\begin{align}
\label{eqn:empirical-fringe-def}
	\fP_{n}^k[{\bf s}] = \frac{1}{n} \sum_{v\in \cT_n} \delta_{\set{F_k(v,\cT_n) \simeq {\bf s}}}. 
\end{align}
Thus, $(\fP_{n}^k)_{n\geq 1}$ can be viewed as a random sequence in $\cM_{\pr}(\bbT^k)$ and we can talk about weak convergence of this sequence, which is the content of (b) and (c) of the definition below. Furthermore, with $k=0$, define the probability measure $\E(\fP_n^0)$ on $\bbT$ via the operation $\E(\fP_n^0)[\vt] = \E(\fP_n^0[\vt])~\forall~\vt \in \bbT$.  

\begin{defn}[Local convergence]
	\label{def:local-weak}
Fix a probability measure $\varpi$ on $\bT$.
	\begin{enumeratea}
\item[(a)] \label{it:fringe-exp} 
We say that a sequence of trees $(\cT_n)_{n\geq 1}$ converges in \emph{expectation}, in the fringe sense, to $\varpi$, if 
 \[\E(\fP_n^{0}) \to \varpi, \quad \text{ on } \cM_{\pr}(\bbT) \quad  \text{ as } n\to\infty. \]
 Denote this convergence by $\TT_n\Efr \varpi$ as $n\to\infty$.
	    \item[(b)] \label{it:fringe-a}  We say that a sequence of trees $(\cT_n)_{n\geq 1}$ converges in the probability sense, in the fringe sense, to $\varpi$, if \[\fP_n^{0} \probc \varpi, \quad \text{ as } n\to\infty. \]
	Denote this convergence by $\TT_n\probfr \varpi$ as $n\to\infty$.
	    \item[(c)] \label{it:fringe-b} We say that a sequence of trees $(\cT_n)_{n\geq 1}$ converges in probability, in the \emph{extended fringe sense}, to a limiting infinite random {\tt sin}-tree $\TT_{\infty}$ if, for all $k\geq 0$,
	  \[\fP_n^k \probc \prob\left(F_k(0,\TT_{\infty}) \in \cdot \right), \qquad \text{ as } n\to\infty. \]
	Denote this convergence by $\TT_n\probcrf \TT_{\infty}$ as $n\to\infty$.
	\end{enumeratea}
\end{defn} 
In an identical fashion, one can define notions of convergence in distribution or almost surely in the fringe, respectively, extended fringe sense. Letting $\varpi_{\infty}(\cdot) = \pr(F_0(0, \cT_{\infty}) = \cdot)$ denote the distribution of the fringe of $\cT_\infty$ on $\bT$, the convergence in (c) above clearly implies convergence in notion (b) with $\varpi =\varpi_{\infty}(\cdot)$. More surprisingly, if the limiting distribution $\varpi$ in (b) has a certain `stationarity' property (defined below), \emph{convergence in the fringe sense implies convergence in the extended fringe sense} as we now describe. Further, under an ``extremality'' condition of the limit objection in (a), convergence in expectation implies convergence in probability as in (b). We need the following definitions:

\begin{defn}[Fringe distribution~\cite{aldous-fringe}] 
\label{fringedef}
We say that a probability measure $\varpi$ on $\bbT$ is a fringe distribution if 
\[
\sum_{\vs} \varpi(\vs) \vQ(\vs, \vt) = \varpi(\vt), \qquad \forall~\vt \in \bbT. 
\]
\end{defn}

It is easy to check that the space of fringe distributions $\cM_{\pr, \fringe}(\bbT) \subseteq \cM_{\pr}(\bT)$ is a convex subspace of the space of probability measure on $\bT$, and so we can talk about extreme points of this convex subspace. The following fundamental theorem is one of the highlights of~\cite{aldous-fringe}:

\begin{thm}[Consequence of fringe convergence \cite{aldous-fringe}]
\label{thm:aldous-efr-pfr}
Fix a fringe distribution $\varpi \in \cM_{\pr, \fringe}(\bbT)$. Suppose that a sequence of trees $(\cT_n)_{n\geq 1}$ converges in the expected fringe sense to $\cT_n \Efr \varpi$ as $n\to\infty$ (recall Definition \ref{def:local-weak}(a)). If $\varpi$ is extremal in the space of fringe measures, then the above convergence in expectation automatically implies convergence in probability in the fringe sense, i.e., $\cT_n \probfr \varpi$ (recall Definition \ref{def:local-weak}(c)).
\end{thm}

The advantage of Theorem~\ref{thm:aldous-efr-pfr} is that, to prove convergence in the probability fringe sense, it is enough to deal with expectations, at least under the extremality of the limit object.  The next result shows that convergence in the probability fringe sense often automatically implies convergence to a limit infinite {\tt sin}-tree. To formulate the statement, we need an additional definition. For any fringe distribution $\varpi$ on $\bbT$, we can uniquely obtain the law $\varpi^{{\rm \sss EF}}$ of a random {\tt sin}-tree $\TT$ with monotone decomposition $(\bar\bt_0(\TT), \bar\bt_1(\TT),\ldots)$ such that, for any $i \in \Zbold_+$ and any $\bar \vt_0, \bar \vt_1, \dots$ in $\bbT$,
\begin{align}
\label{ftoef}
\varpi^{{\rm \sss EF}}((\bar\bt_0(\TT), \bar\bt_1(\TT),\ldots, \bar\bt_i(\TT)) = (\bar \vt_0,\vt_1,\dots,\vt_i)) 
= \varpi(\vt_i) \prod_{j=1}^{i} \vQ(\vt_i,\vt_{i-1}),
\end{align}
where the product is taken to be 1 when $i=0$. The following lemma follows by adapting the proof of~\cite{aldous-fringe}*{Propositions 10 and 11}.

\begin{lemma}[Convergence to sin-tree]
\label{ftoeflemma}
{\it Suppose that a sequence of trees $(\cT_n)_{n\geq 1}$ converges in probability, in the fringe sense, to $\varpi$. Moreover, suppose that $\varpi$ is a fringe distribution in the sense of Definition~\ref{fringedef}. Then $(\cT_n)_{n\geq 1}$ converges in probability, in the {\em extended} fringe sense, to a limiting infinite random sin-tree $\TT_{\infty}$ whose law $\varpi^{{\rm \sss EF}}$ is uniquely obtained from $\varpi$ via~\eqref{ftoef}.}
\end{lemma}

Fringe convergence and extended fringe convergence imply convergence of functionals, such as the degree distribution. For example, with $\cT_{\varpi} \sim \varpi$ with root $0$, convergence in the sense of~\eqref{it:fringe-a} in particular implies that, for any $k\geq 0$, 
\begin{align}
\label{eqn:deg-convg-fr}
	\frac1n\,\#\set{v\in \TT_n\colon\, d(v,  \TT_n) = k+1} \convp \prob(d(0,\cT_{\varpi})=k).
\end{align}
However, both convergences give much more information about the asymptotic properties of $(\cT_n)_{n\geq 1}$ beyond its degree distribution.


\section{Alternate description of \texorpdfstring{$\BP_{\Sr}$}{BP-Mac} via Edge branching process }
\label{edgedesc}

The branching process $\BP_{\Sr}$, whilst playing a key role in the description and proof of the local limit, is less tractable when trying to extract limit information such as the behavior of the degree tail exponent. The goal of this section is to describe an alternative probabilistic construction of the same object, where the corresponding \emph{offspring} distribution of $\BP_{\Sr}$ is shown to be equivalent to a \emph{branching process with immigration}. Thus, embedded within the original branching process is another branching process. Similar phenomena have been noticed for the original preferential attachment model as well as more recent models on preferential attachment with delay \cite{BBDS04_macro}. 

Consider the branching process $\BP_{\Sr}$ in Definition~\ref{def:limit-bp-macro} (we suppress the dependence on $\delta$ to ease notation). With some abuse of notation, write $(\cP_{\Sr}(t))_{t\geq 0}$ to denote the process representing the number of children of the root, namely, for any $t\geq 0$, 
\[
\cP_{\Sr}(t) = \#\set{i\in\mathbb{N}\colon\, \sum_{j=1}^i \cE_{(j-1)\leadsto j} \leq t}.
\] 
By Theorem~\ref{thm:main-local} below, the limit degree distribution $p_{\Sr}$ is just the probability mass fuction of the random variable $\cP_{\Sr}(T_1)+1$, where $T_1\sim\exp(1)$ is independent of $\cP_{\Sr}$. Handling $\BP_{\Sr}$ and $\cP_{\Sr}$ directly by using the hazard rate descriptions in~\eqref{eqn:h01-hazard} and~\eqref{eqn:h-gen-hazard} is non-trivial, as the inter-arrival times of $\cP_{\Sr}$ are intertwined in a way that makes understanding probabilistic properties, such as the tail exponent of the degree distribution, intractable. 

We now present a ``dual'' description of the \emph{offspring point process} $\cP_{\Sr}$ in terms of another branching process $\BP_{\Sr,\ve}$, viewing this point process as a mechanism where \emph{edges reproduce}. This construction gives rise to the independence we need to analyze various asymptotics of its growth rate, and subsequently the offspring of the original process. Thus, to summarize, the \emph{offspring distribution} of the limit branching processes has \emph{another branching process} embedded in it:

\begin{defn}[Edge branching process $\BP_{\Sr, \ve}$]
\label{def:bpmave}
Consider the following branching process: 
\begin{enumerate}
\item[(a)]
At time $0$, the population consists of $1$ individual $\tilde{v}_0$. This individual reproduces according to an inhomogeneous Poisson point process with rate function 
\begin{equation}
	\label{eqn:rdel-one}
r_\ve^{\sss(1)}(t) = \frac{1+\delta}{1+\tfrac12\delta}\,(1-\eee^{-t}), \qquad t\geq 0.
\end{equation}
\item[(b)]
\label{item:rt} 
For all vertices other than the root, at time $t \ge 0$ after its birth, an existing (non-root) vertex $\ve$ reproduces according to an inhomogeneous Poisson point process with rate function 
    \begin{equation}
    \label{eqn:rdel-typical}
        r_\ve(t) = \frac{1}{1+\tfrac12\delta}\,(1-\eee^{-t}),\qquad t\geq 0,
    \end{equation} 
giving birth to a new vertex. We will sometimes refer to the vertices in this branching process as ``edges'' (towards a new individual).
\end{enumerate}
Denote the \emph{entire set of descendants}  of $\tilde{v}_0$ at time $t$ by $(\BP_{\Sr,\ve}(t))_{t \ge 0}$. 
\end{defn}
Note that other than the root, all the other vertices have the same offspring distribution given by a Poisson point process with intensity function $r_\ve$ as in \eqref{eqn:rdel-typical}.  We set up some notation to facilitate the analysis of this process for later use. Let $\mvzeta_\ve$ be the Poisson point process with rate function $r_\ve$ as in \eqref{eqn:rdel-typical}. 

\begin{enumerate}
    \item[(a)] Let $\BP^\circ_{\ve}$ denote a (continuous-time) branching process started with one individual at time $t=0$ and with offspring distribution  $\mu_{\mvzeta_{\ve}}$. Let $(\BP^\circ_{\ve,i})_{i \in \mathbb{N}_0}$ be i.i.d.\ copies of $\BP^\circ_{\ve}$ with respective sizes given by $(|\BP^\circ_{\ve,i}(t)|)_{i \in \mathbb{N}_0}$.
    \item[(b)] Let $(\chi(t))_{t \ge 0}$ denote a Poisson point process with rate $r_\ve^{\sss(1)}(t)$ as in \eqref{eqn:rdel-one}. Let $(\theta_i)_{i \in \mathbb{N}_0}$ denote the reproduction times of $\chi$, with $\theta_0=0$ (that is, $\chi(0)=1$).
\end{enumerate}

Consider the branching process with immigration $\BP_{\ve}$, starting with one immigrant, where the progeny of each new individual that enters the system according to (an independent copy of) $\BP^\circ_{\ve}$, and there is an additional incoming stream of individuals into the population at epochs of $\chi$. This describes the evolution of $\BP_{\Sr,\ve}(t)$. The size of this branching process is thus given by
\begin{equation}
    |\BP_{\Sr,\ve}(t)| := \sum_{i=0}^{\chi(t)-1} |\BP^\circ_{\ve,i}(t - \theta_i)|, \qquad t \ge 0.
\end{equation}
Define the Malthusian rate of growth of the edge process $\BP^\circ_{\ve}$ as the unique solution to the equation 
\begin{align}
        \label{eqn:malthus-edge}
        \int_0^\infty \eee^{-\gl t} r_\ve(\dd t) = 1. 
\end{align}
Inserting \eqref{eqn:rdel-typical}, we get that $\gl = \gl(\delta)$ satisfies the quadratic equation 
\begin{equation}
\label{eq:lambdaeta}
\lambda^2 + \lambda = \frac{1}{1+\tfrac12\delta}.
\end{equation}
In particular, $\lambda(\delta) = 1/\phi(\delta)$ as defined in \eqref{phidef}, which gives an alternative interpretation of $\phi(\delta)$. With the above choice of the Malthusian rate of growth $\lambda$, define the random variable 
\begin{equation}
    \label{eqn:hatmv}
    \hat\mvzeta_{\ve}(\gl) = \int_0^\infty \eee^{-\gl t} \mvzeta_{\ve}(\dd t). 
\end{equation}

The following finiteness bounds will be needed in Section \ref{sec:proofs}:

\begin{lem}[Moment bounds]
\label{prop:zeta-bounds}
$\E([\hat\mvzeta_{\ve}(\gl)]^l) < \infty$ for every $l \geq 1$. This implies that for every $l\geq 1$, there exist constants $H_l^{\circ}, H_l^\prime < \infty$, such that
\begin{equation}
\label{eqn:sup-bound}
	\sup_{t\geq 0} \E\left[\big[\ee^{-\lambda t} |\BP^\circ_{\ve}(t)|\big]^l\right] \leq H_l^{\circ}, 
	\qquad  \Rightarrow \sup_{t\geq 0} \E\left[\big[\ee^{-\lambda t} |\BP_{\Sr,\ve}(t)|\big]^l\right] \leq H_l^\prime. 
\end{equation}
\end{lem}

\begin{proof}
We start with the first assertion. For this it is enough to show the finiteness of the $l$-th cumulant. For simplicity, let $c = \frac{1}{1 + \frac{1}{2}\delta}$, so the rate function is $r_{\mathbf{e}}(t) = c(1 - \mathrm{e}^{-t})$.   For a Poisson point process with rate function $r_{\ve}$, by Campbell's formula, the $l$-th cumulant $\kappa_l$ of the random variable $X = \int_{0}^{\infty} g(t) \mvzeta(\mathrm{d}t)$ is given by:
\[
\kappa_l = \int_{0}^{\infty} (g(t))^l r_{\mathbf{e}}(t) \mathrm{d}t.
\]
Computing the $l$-th cumulant for $\hat{\mvzeta}_{\mathbf{e}}(\lambda)$ gives
\[
\kappa_l = c \left( \frac{1}{l\lambda} - \frac{1}{l\lambda + 1} \right) = \frac{c}{l\lambda(l\lambda + 1)} < \infty. 
\]
The first assertion of \eqref{eqn:sup-bound} follows from the equivalence of the existence of moments of $\hat\mvzeta_{\ve}(\gl)$ and the corresponding  moments of the normalized branching process, see~\cite[Theorem 4]{mori2019moments}. To prove the second assertion, note that by H\"older's inequality $\sum_{i=0}^n |a_ib_i| \le \left(\sum_{i=0}^n |a_i|^p\right)^{1/p}\left(\sum_{i=0}^n |b_i|^q\right)^{1/q}$, with $n=\chi(t)-1$, $p=\frac{l}{l-1}, q = l$, $a_i = \ee^{-\gl \frac{l-1}{l} \theta_i}$, $b_i = \ee^{-\gl \frac{1}{l} \theta_i} \ee^{-\gl (t-\theta_i)}|\BP^{\circ}_{\ve,i}(t - \theta_i)|$, 
\begin{align*}
\E\left(\left[\ee^{-\lambda t}|\BP_{\Sr,\ve}(t)|\right]^l\right) 
&= \E\left(\left[\sum_{i=0}^{\chi(t)-1} \ee^{-\gl \theta_i} \ee^{-\gl (t-\theta_i)}|\BP^{\circ}_{\ve,i}(t - \theta_i)|\right]^l\right)\\
&\le \E\left(\left[\sum_{i=0}^{\chi(t)-1} \ee^{-\gl \theta_i}\right]^{l-1}\left[\sum_{i=0}^{\chi(t)-1} 
\ee^{-\gl \theta_i}\left(\ee^{-\lambda (t-\theta_i)}|\BP^{\circ}_{\ve,i}(t - \theta_i)|\right)^l\right]\right)\\
& \le H^\circ_l \,\E\left(\left[\sum_{i=0}^{\chi(t)-1} \ee^{-\gl \theta_i}\right]^{l}\right) \le
 H^\circ_l \,\E\left(\left[\int_0^\infty \ee^{-\gl t}\chi(\dd t)\right]^l\right),
\end{align*}
where we have applied the independence of $\chi$ and $(\BP^{\circ}_{\ve,i})_{i \in \mathbb{N}_0}$ in the second inequality. Noting that, by the form of $r_{\ve}^{\sss(1)}$, 
\[\E\left(\left[\int_0^\infty \mathrm{e}^{-\gl t}\chi(dt)\right]^l\right) \leq (1+\delta)^l \E([\hat\mvzeta_{\ve}(\gl)]^l) < \infty,\]
completes the proof. 
\end{proof}

Next, we explain the connection between the edge branching process and the limit degree distribution.
The following proposition will be needed in Section \ref{sec:proofs}.

\begin{prop}[Equivalence]
\label{prop:edge-equiv-bp} 
The following distributional equivalence holds:
    \begin{equation}
    \label{eq:degeq}
        (\cP_{\Sr}(t)+1)_{t\geq 0} \stackrel{d}{=} (|\BP_{\ve}(t)|)_{t\geq 0}.
    \end{equation}
Thus, $(\BP_{\Sr,\ve}(t))_{t \ge 0} \stackrel{d}{=} (\BP_{\Sr}(t))_{t \ge 0}$. Consequently, with $D_{\Sr} = \cP_{\Sr}(T_1) +1$ denoting the limit degree random variable in Theorem~\ref{thm:main-local} with distribution $p_{\Sr}(\cdot)$, 
\[
D_{\Sr} \stackrel{d}{=} |\BP_{\ve}(T_1)|,
\] 
where $T_1 \sim \exp(1)$ is independent of $\BP_{\ve}$.   
\end{prop}

\begin{proof}[Proof of Proposition~\ref{prop:edge-equiv-bp}]
Let $N(t) = \abs{\BP_{\ve}(t)}$ denote the total number of individuals in the edge branching process at time $t$. We aim to show that the counting process $(N(t))_{t \ge 0}$ has the same probability law as $(\cP_{\Sr}(t) + 1)_{t \ge 0}$. Since both processes start with exactly $1$ individual at time $t=0$, it suffices to demonstrate that the conditional hazard rates for the inter-arrival times of new births are identical for both processes, given their history.

We begin by considering the initial birth. At time $t=0$, the population consists solely of the root vertex $\tilde{v}_0$, which generates its first offspring according to an inhomogeneous Poisson point process with the rate function
\eqn{
    r_\ve^{\sss(1)}(t) = \frac{1+\delta}{1+\tfrac12\delta}\,(1-\eee^{-t}), \qquad t\geq 0.
}
The time until this first arrival naturally has a hazard rate equal to this intensity function, exactly matching the hazard function $h_{0 \leadsto 1}(x)$ defined for $\sE_{0 \leadsto 1}$ in the process $\cP_{\Sr}$.

To see that this correspondence holds for all subsequent births, assume the process currently has $k$ individuals (the root and $k-1$ descendants) at time $t$. Let $\sigma_0 = 0$ be the birth time of the root, and let $\sigma_1 < \sigma_2 < \dots < \sigma_{k-1} \le t$ be the realized birth times of the $k-1$ descendants. Because every individual in $\BP_{\ve}$ reproduces independently, we can use the superposition property of independent Poisson processes. The aggregate intensity (or hazard rate) for the next birth in the entire population is simply the sum of the individual reproduction rates of all existing vertices evaluated at their respective current ages. The root vertex reproduces at rate $r_\ve^{\sss(1)}(t)$, while any other vertex $j$ (born at time $\sigma_j$) reproduces at rate
\eqn{
    r_\ve(t - \sigma_j) = \frac{1}{1+\tfrac12\delta}\,(1-\eee^{-(t-\sigma_j)}).
}

Summing these individual rates yields the aggregate hazard rate $\lambda_k(t)$ for the next birth at time $t > \sigma_{k-1}$, given by
\eqn{
    \lambda_k(t) = r_\ve^{\sss(1)}(t) + \sum_{j=1}^{k-1} r_\ve(t - \sigma_j).
}
To align this with the inter-arrival hazard function of $\cP_{\Sr}$ in \eqref{eqn:h-gen-hazard}, we express this total rate in terms of $x$, the time elapsed since the most recent birth. Setting $x = t - \sigma_{k-1}$ (so that $t = x + \sigma_{k-1}$) and substituting the explicit formulas for the rate functions $r_\ve^{\sss(1)}$ and $r_\ve$, we obtain
\eqn{
    \lambda_k(x + \sigma_{k-1}) = \frac{1+\delta}{1+\tfrac12\delta}(1- \eee^{-(x+\sigma_{k-1})}) + \sum_{j=1}^{k-1} \frac{1}{1+\tfrac12\delta}(1- \eee^{-(\sigma_{k-1} - \sigma_j + x)}).
}
Factoring out the constant in the sum yields
\eqn{
    \lambda_k(x + \sigma_{k-1}) = \frac{1+\delta}{1+\tfrac12\delta}(1- \eee^{-(x+\sigma_{k-1})}) 
    + \frac{1}{1+\tfrac12\delta}\sum_{j=1}^{k-1}(1- \eee^{- (\sigma_{k-1} - \sigma_j +x)})
    = h_{(k-1)\leadsto k}(x),
}
where, as  in \eqref{eqn:h-gen-hazard}, $h_{(k-1)\leadsto k}(x)$ is the hazard rate defining the inter-arrival time $\sE_{(k-1) \leadsto k}$ of the point process $\cP_{\Sr}$, and we recall that, in both contexts, $\sigma_j$ represents the absolute time of the $j$-th arrival equal to $\sum_{i=1}^j \sE_{(i-1) \leadsto i}$. Because the initial states are identical and the sequence of conditional hazard rates for all subsequent jump times match exactly, the joint distribution of the inter-arrival times for $\abs{\BP_{\ve}(t)}$ is identical to that of $\cP_{\Sr}(t) + 1$. Thus, the two processes are distributionally equivalent, i.e.,
\eqn{
    (\cP_{\Sr}(t)+1)_{\geq 0} \equald (\abs{\BP_{\ve}(t)})_{t\geq 0}.
}
\end{proof}

The above construction will play a crucial role towards understanding the tail behavior of the limiting degree distribution. For later use, we write $\cP_{\ve}$ for the edge offspring point process, namely, the inhomogeneous Poisson point process with rate function $r_\ve$ as defined above. 

For proving extremity of the fringe distribution, by Aldous \cite{aldous-fringe}, the following is enough:

\begin{lem}[Extremity of limiting fringe distribution]
	\label{lem:mean-one}
Recall the edge branching process $\BP_\ve$, and let $T_1$ denote an independent exponential mean-$1$ random variable. 
Then $\E[|\BP_\ve(T_1)|]=2$, and hence $\E[\cP_\Sr(T_1)]=1$.
\end{lem}

\begin{proof}
Recall that, for $\delta \neq 0$, the root has a different rate $r^{\sss(1)}_\ve$ than the rest of the vertices, which have rates $r_\ve$. Let $m(t) = \E[|\BP(t)|]$, and let $\fm(t) = \eee^{-t}m(t)$. Let $\BP^{\sss(0)}(\cdot)$ denote the branching process distribution of all subsequent generations from the root, and let $\fm^{\sss 0}(t) = \eee^{-t} \E[\BP^{\sss(0)}(t)]$. Note that, by construction, 
\[
\fm(t) = \eee^{-t} + \int_0^t \eee^{-s} r_\ve^{\sss(1)}(s)\fm^{\sss(0)}(t-s)\, \dd s.  
\]
Similarly, 
\[
\fm^{\sss (0)}(t) = \eee^{-t} + \int_0^t \eee^{-s} r_\ve(s) \fm^{\sss(0)}(t-s)\, \dd s.
\]
Using \eqref{eqn:rdel-typical}, we can easily verify that 
\[
\int_0^\infty \fm^{\sss(0)}(t)\, \dd t = 1+ \frac{1}{2(1+\tfrac12\delta)} \int_0^\infty \fm^{\sss(0)}(t)\,\dd t.
\]
This implies that 
\[
\int_0^\infty \fm^{\sss(0)}(t)\, \dd t = \frac{2+\delta}{1+\delta}. 
\]
Thus, 
\[
\E[|\BP_{\ve}(T_1)|] = \int_0^\infty \fm(t)\,\dd t = 1 + \int_0^\infty \eee^{-s} r_\ve^{\sss(1)}(s)\, \dd s 
\times \frac{2+\delta}{1+\delta} = 2.
\]
Since the expected degree is $\E[\cP_\Sr(T_1)] = \E[|\BP_\ve(T_1)| -1]$, this settles the claim.
\end{proof}


\section{Proofs}
\label{sec:proofs}

In this section, we complete the proofs of our main results. We prove Theorem~\ref{thm:main-local} in Section \ref{sec:proof-main-local},
Theorem~\ref{thm:tail} in Section \ref{sec:tail-exp}, and Theorem~\ref{degree_growth} in Section \ref{sec:max-deg-proof}.


\subsection{Proof of local convergence: Theorem~\ref{thm:main-local}}
\label{sec:proof-main-local}

The main goal of this section is the following proposition. In its statement, and for trees $\vs,\vt$we write $\vs=\vt$ when they are equal aas ordered (Ulam-Harris) trees.

\begin{prop}[Convergence of root-neighbourhood tree distribution]
\label{prop:expec-fringe}
For every $\bt \in \bbT$ and $r \in \bN$,
\begin{equation}
\label{eqn:limmatch}
\lim_{n\to\infty} \pr(\bbT^r_{o_n} = \bt) = \pr(\bbT^r_\emptyset = \bt),
\end{equation}
where $o_n$ is a vertex drawn uniformly at random from $[n]$, $\bbT^r_{o_n}$ is the $r$-neighbourhood of $o_n$, and
$\bbT^r_\phi$ is the $r$-neighbourhood of $\emptyset$ in the {\tt sin}-tree with distribution $\varpi_{\Sr,\delta}$. 
\end{prop}

\noindent {\bf Completing the proof of Theorem \ref{thm:main-local}.} By Lemma \ref{lem:mean-one}, $\varpi_{\Sr,\delta}$ is a fringe distribution, while \cite{aldous-fringe} shows that this limit distribution, derived from a continuous-time branching process, is an extremal fringe distribution. Thus, Theorem \ref{thm:aldous-efr-pfr} completes the proof. 

\begin{proof}[Proof of Proposition \ref{prop:expec-fringe}:]
The proof proceeds in 6 steps.

\medskip\noindent
{\bf 1.}
Let $\bar\bt = (\bt,(a_v)_{v \in V(\bt)})$ be the marked version of the tree $\bt$, with marks $a_v \in [0,1]$ for all vertices $v \in V(\bt)$. Let $\bbT^r_{o_n}$ be the forward $r$-neighbourhood of $o_n$ in $\bar\bbT$ (i.e., consisting of $o_n$ and all its younger descendants). Let $\bar\bbT^r_{o_n}$ be the marked forward $r$-neighbourhood of $o_n$ in $\bar\bbT$, obtained by setting
\[
\bar\bbT^r_v = \big(\bbT^r_v, (\tfrac{u}{n})_{u \in V(\bbT^r_v)}\big), \qquad v \in [n].
\]    
Then
\eqn{
\label{marked-tree}
\pr(\bar\bbT^r_{o_n} = \bar\bt) = \pr\big(\bbT^r_{o_n} = \bt, \tau_\omega = \lceil n a_\omega \rceil 
\,\,\forall\,\omega \in V(\bt)\big),
}
where $\tau_\omega$ is the vertex in $\bbT^r_{o_n}$ corresponding to vertex $\omega$ in $\bt$, representing the time at which vertex $\omega$ enters and connects. We define the vertex weight
	\eqn{
	\label{vertex-weight}
	W_v(t)=\theta(v,t).
	}
The right-hand side of \eqref{marked-tree} equals (see Figure \ref{fig:discrete_tree})
\begin{equation}
\label{eqn:product}
\prod_{s \in \{\tau_u\colon u \in V(\bt)\}} \frac{W_{\tau_{p(s)}}(\tau_s-1)}{D^\delta(\tau_s-1)} 
\quad \prod_{ {s \notin \{\tau_u\colon u \in V(\bt)\}} \atop {s>\lceil na_\emptyset \rceil} } 
\Big(1- \sum_{ {\omega \in V(\bt):} \atop {\tau_{\omega}<s} }
\frac{W_{\tau_\omega}(s-1)}{D^\delta(s-1)}\Big),
\end{equation}
where $p(s)$ is the parent of $s$ in $\bt$ and 
	\[
	D^\delta(u) = u(u+1)(1+\tfrac12\delta)
	\]
is the total weight at time \(u\). The first product represents the probability for the presence of the edges in $\bt$, while the second product represents the probability for the absence of any further edges in $\bt$. In Steps 2--4 we will compute $\pr(\bar\bbT^r_{o_n} \ch{=} \bar\bt)$ in terms of the marks. 

\medskip\noindent
{\bf 2.} The first product in \eqref{eqn:product} equals, where we recall that $d(u,\vt)$ is the degree of vertex $u$ in $\vt$,
\begin{equation}
\label{eqn:1stprod}
\prod_{u \in B_{r-1}(\emptyset)} \prod_{k=1}^{d(u,\vt)} \frac{W_{\tau_u}(\tau_{uk}-1)}{D^\delta(\tau_{uk}-1)},  
\end{equation} 
where $B_\ell(\emptyset)$ is the ball of radius $\ell$ around $\emptyset$, $d(u,\vt)$ is the number of offspring of vertex $u$, and $uk$ is the $k$-th offspring of vertex $u$ in the Ulam-Harris ordering. 

\input{discrete_tree}

\medskip\noindent
Using that $W_{\tau_u}(\tau_{uk}-1)=\sum_{m=1}^k (\tau_{um} - \tau_{u(m-1)})\,(m+\delta)$, we see that the last product in \eqref{eqn:1stprod} can be written out as

	\eqn{
	\prod_{u \in B_{r-1}(\emptyset)}I(u),
	}
where	
\begin{equation}
\label{eqn:I}
I(u) = (1+\tfrac12\delta)^{-d(u,\vt)} \prod_{k=1}^{d(u,\vt)} 
\frac{\sum_{m=1}^k (\tau_{um} - \tau_{u(m-1)})\,(m+\delta)}{(\tau_{uk}-1)\tau_{uk}}.
\end{equation}
The numerator of $I(u)$ equals, via telescoping and with the convention $\tau_{u0} = \tau_u$,
\[
\begin{aligned}
\sum_{m=1}^k m\tau_{um} - \sum_{m=0}^{k-1} (m+1)\tau_{um} + (\tau_{uk}-\tau_u)\delta
&= \sum_{m=0}^k (\tau_{uk}-\tau_{um}) + (\tau_{uk}-\tau_u)\delta\\
&= \sum_{m=0}^k (\lceil na_{uk} \rceil - \lceil na_{um} \rceil) + (\lceil na_{uk} \rceil - \lceil na_u \rceil)\delta\\
&= n \left[\sum_{m=0}^k (a_{uk}-a_{um}) + (a_{uk}-a_u)\delta\right] \big(1+O(n^{-1})\big).
\end{aligned}
\]
The denominator of $I(u)$ equals
\[
\lceil na_{uk} \rceil (\lceil na_{uk} \rceil -1) = n^2 a^2_{uk} (1+O(n^{-1})).
\] 
Substitution into \eqref{eqn:I} yields
\[
I(u) = (1+\tfrac12\delta)^{-d(u,\vt)}\, n^{-d(u,\vt)}\, \big(1+O(n^{-1})\big)^{d(u,\vt)} \prod_{k=1}^{d(u,\vt)} 
\frac{\sum_{m=0}^k (a_{uk}-a_{um}) + (a_{uk}-a_u)\delta}{a^2_{uk}}.
\]
Substitution of this expression into \eqref{eqn:1stprod} gives
\begin{equation}
\label{eqn:factor1} 
(1+\tfrac12\delta)^{-|E(\bt)|}\, n^{-|E(\bt)|}\,\big(1+O(n^{-1})\big)^{|E(\bt)|} 
\prod_{u \in B_{r-1}(\emptyset)} \prod_{k=1}^{d(u,\vt)}
\frac{\sum_{m=0}^k (a_{uk}-a_{um}) + (a_{uk}-a_u)\delta}{a^2_{uk}}.
\end{equation}
 
\medskip\noindent
{\bf 3.} Since we investigate the {\em forward} neighborhood of the root, the second product in \eqref{eqn:product} is restricted to $s>\lceil na_\emptyset \rceil$, and thus equals 
\begin{equation}
\label{eqn:2ndprod}
\begin{aligned}
&\prod_{ {s \notin \{\tau_u\colon u \in V(\bt)\}} \atop {s>\lceil na_\emptyset \rceil} }
\exp\left[-\sum_{ {\omega \in V(\bt):} \atop {\tau_\omega<s} }
\frac{W_{\tau_\omega}(s-1)}{s(s-1)(1+\tfrac12\delta)} + O(s^{-2})\right]\\
&\qquad = \exp\left[-(1+\tfrac12\delta)^{-1} \sum_{ {s \notin \{\tau_u\colon u \in V(\bt)\}} \atop {s>\lceil na_\emptyset \rceil} } 
\left\{\sum_{ {\omega \in V(\bt):} \atop {\tau_\omega<s}  } \frac{W_{\tau_\omega}(s-1)}{s(s-1)} + O(s^{-2})\right\}\right]\\
&\qquad = \exp\left[-(1+\tfrac12\delta)^{-1} \sum_{s>\lceil na_\emptyset \rceil} 
\sum_{ {\omega \in V(\bt):} \atop {\tau_\omega<s}  } \frac{W_{\tau_\omega}(s-1)}{s(s-1)} \right.\\
&\qquad\qquad\qquad\qquad \left.+ (1+\tfrac12\delta)^{-1} \sum_{s\in\left\{ \tau_{u}: u\in V(\bt) \right\}}\sum_{ {\omega \in V(\bt):} \atop {\tau_\omega<s}  } \frac{W_{\tau_\omega}(s-1)}{s(s-1)} + O((na_\emptyset)^{-1}) \right]\\
&\qquad=\exp\left[-(1+\tfrac12\delta)^{-1} \sum_{s>\lceil na_\emptyset \rceil} 
\sum_{ {\omega \in V(\bt):} \atop {\tau_\omega<s}  } \frac{W_{\tau_\omega}(s-1)}{s(s-1)} + O((na_\emptyset)^{-1}) \right],
\end{aligned}
\end{equation}
where, in the  second equality, we remove the constraint that $s \notin \{\tau_u\colon u \in V(\bt)\}$, thereby adding only finitely many terms, and in the third inequality, we use that  $\sum_{\omega \in V(\bt)\colon\, \tau_\omega<s} W_{\tau_\omega}(s-1) \leq s |V(\bt)|^2(1+\delta) = O(s)$. The double sum in the exponent equals
\begin{equation}
\label{eqn:II}
\begin{aligned}
II &= \sum_{u \in B_{r-1}(\emptyset)} 
\Bigg[\sum_{k=1}^{d(u,\vt)}  \sum_{s=\tau_{u(k-1)}+1}^{\tau_{uk}}  
\frac{S_{u(k-1)} + (s-\tau_{u(k-1)})(k+\delta)}{s(s-1)}\\
&\qquad \qquad \qquad \qquad + \sum_{s=\tau_{u d(u,\vt)}+1}^n  
\frac{S_{ud(u,\vt)} + (s-\tau_{u d(u,\vt)})(d(u,\vt)+1+\delta)}{s(s-1)}\Bigg],
\end{aligned}
\end{equation}
where we abbreviate 
\eqn{
\label{Auk-def}
S_{uk} = \sum_{m=1}^k (\tau_m-\tau_{m-1})(m+\delta).
}
Since $\tfrac{1}{s(s-1)} = \tfrac{1}{s-1}-\tfrac{1}{s}$, the sums over $s$ can be carried out, to give that the term between square brackets equals
\begin{equation}
\label{eqn:II1}
\begin{aligned}
&\sum_{k=1}^{d(u,\vt)} \left(S_{u(k-1)} - \tau_{u(k-1)} (k+\delta)\right) 
\left[\frac{1}{\tau_{u(k-1)}} - \frac{1}{\tau_{uk}}\right]\\
&\qquad + \left(S_{ud(u,\vt)} - \tau_{u d(u,\vt)} (d(u,\vt)+1+\delta)\right) 
\left[\frac{1}{\tau_{u d(u,\vt)}} - \frac{1}{n}\right]\\
&\qquad + \sum_{k=1}^{d(u,\vt)} (k+\delta) \left[\log (\tau_{uk}-1) - \log \tau_{u(k-1)} + O(n^{-1}))\right]\\
&\qquad \qquad \qquad + (d(u,\vt)+1+\delta) \left[\log(n-1) - \log\tau_{u d(u,\vt)} + O(n^{-1})) \right].
\end{aligned}
\end{equation}
By further telescoping and \eqref{Auk-def}, we have
\[
\begin{aligned}
S_{u(k-1)} - \tau_{u(k-1)} (k+\delta) 
&= -  \sum_{m=0}^{k-1} \tau_{um} - \tau_u \delta,\\
S_{ud(u,\vt)} - \tau_{u d(u,\vt)} (d(u,\vt)+1+\delta) 
&= - \sum_{m=0}^{d(u,\vt)} \tau_{um} - \tau_u \delta.
\end{aligned}  
\]
With the convention $\tau_{u(d(u,\vt)+1)} = n$, the first two terms in \eqref{eqn:II1} simplify to
\[
\begin{aligned}
&- \sum_{k=1}^{d(u,\vt)} \left(\sum_{m=0}^{k-1} \tau_{um} + \tau_u \delta\right) 
\left[\frac{1}{\tau_{u(k-1)}} - \frac{1}{\tau_{uk}}\right]
-  \left(\sum_{m=0}^{d(u,\vt)} \tau_{um} + \tau_u \delta\right) 
\left[\frac{1}{\tau_{u d(u,\vt)}} - \frac{1}{n}\right]\\
&\qquad = \tau_u \delta\left[\frac{1}{n}-\frac{1}{\tau_u}\right]
+ \sum_{k=1}^{d(u,\vt)} \sum_{m=0}^{k-1} \tau_{um} \left[\frac{1}{\tau_{uk}} - \frac{1}{\tau_{u(k-1)}}\right]
- \left[\frac{1}{\tau_{u d(u,\vt)}} -\frac{1}{n}\right] \sum_{m=0}^{d(u,\vt)} \tau_{um}\\
&\qquad = \tau_u \delta\left[\frac{1}{n}-\frac{1}{\tau_u}\right]
+ \sum_{k=1}^{d(u,\vt)+1} \sum_{m=0}^{k-1} \tau_{um} \left[\frac{1}{\tau_{uk}} - \frac{1}{\tau_{u(k-1)}}\right]\\
&\qquad = \tau_u \delta\left[\frac{1}{n}-\frac{1}{\tau_u}\right]
+ \sum_{m=0}^{d(u,\vt)} \tau_{um} \left[\frac{1}{n} - \frac{1}{\tau_{um}}\right]
\end{aligned}
\]
and the last two terms in \eqref{eqn:II1} simplify to
\[
- \left[\delta \log \left(\frac{\tau_u}{n}\right) + \sum_{k=0}^{d(u,\vt)} \log \left(\frac{\tau_{uk}}{n}\right)\right] + O(n^{-1}).
\]
Collecting terms from $II$ and inserting $\tau_{um} = \lceil n a_{um} \rceil = n a_{um} (1+O(n^{-1}))$ into the sums, we get for \eqref{eqn:2ndprod} the expression
\eqan{
\label{eqn:factor2}
&\exp\Bigg[-(1+\tfrac12\delta)^{-1}\, \big(1+O(n^{-1})\big)\\ 
&\qquad \qquad \times \sum_{u \in B_{r-1}(\emptyset)} 
\Bigg(\delta a_u\left(1-\frac{1}{a_u}\right) + \sum_{m=0}^{d(u,\vt)} a_{um} \left(1-\frac{1}{a_{um}}\right)
- \left[\delta \log a_u + \sum_{k=0}^{d(u,\vt)} \log a_{uk}\right] \Bigg)\Bigg].\nonumber
}

\medskip\noindent
{\bf 4.} 
Combining \eqref{eqn:product}, \eqref{eqn:factor1} and \eqref{eqn:factor2}, we obtain
\begin{equation}
\label{eqn:discrrepr}
\begin{aligned}
&\lim_{n\to\infty} n^{|V(\bt)|} \pr(\bar\bbT^r_{o_n} = \bar\bt)
= \prod_{u \in B_{r-1}(\emptyset)} \prod_{k=1}^{d(u,\vt)}
\frac{\delta(a_{uk}-a_u)+\sum_{m=0}^k (a_{uk}-a_{um})}{(1+\tfrac12\delta)\, a^2_{uk}}\\
&\quad \times \exp\Bigg[-(1+\tfrac12\delta)^{-1}
\sum_{u \in B_{r-1}(\emptyset)} 
\Bigg(\delta\big[(a_u-1)-\log a_u\big] + \sum_{m=0}^{d(u,\vt)} \big[(a_{um}-1) - \log a_{um}\big]\Bigg)\Bigg].
\end{aligned}
\end{equation}
In Steps 5--6 we will compute the Radon-Nikodym derivative of $\pr(\bar\bbT^r_\emptyset = \bar\bt)$ with respect to the marks and show, via a change of time scale, that the latter matches the right-hand side of \eqref{eqn:discrrepr}. After summing out, respectively, integrating out the marks, we get the claim in \eqref{eqn:limmatch}. 

\medskip\noindent
{\bf 5.}
Define
\[
\pr(\bar\bbT^r_\emptyset = \bar\bt) = \frac{\pr\big(\bbT^r_\emptyset = \bt, 
A_\omega \in \dd a_\omega\,\,\forall\,\omega \in V(\bt)\big)}{\prod_{\omega \in V(\bt)} \dd a_\omega}
\]
with $A_\omega$ the mark of the vertex in $\bbT^r_\emptyset$ corresponding to vertex $\omega$ in $\bt$. Write
\[
\begin{aligned}
&\pr\big(\bbT^r_\emptyset = \bt, A_\omega \in \dd a_\omega\,\,\forall\,\omega \in V(\bt)\big)\\
&\hspace{2cm}= \dd a_{\emptyset} \prod_{u \in B_{r-1}(\emptyset)} 
\pr\big(\partial B_1(u) = \partial B_1(\omega_u),
A_\zeta \in \dd a_\zeta\,\,\forall\,\zeta \in \partial B_1(u)\mid A_{u}=a_{u}\big),
\end{aligned}
\]
with $\omega_u$ the vertex in $\bbT^r_\emptyset$ corresponding to vertex $u\in V(\bt)$. We have $\pr(A_{\emptyset} \in \dd a_{\emptyset}) = \dd a_{\emptyset}$ and
\[
\begin{aligned}
&\pr\big(\partial B_1(u) = \partial B_1(\omega_u), A_\zeta \in \dd a_\zeta\,\,\forall\,\zeta \in B_1(u) \mid A_u = a_u\big)\\
&= \left[\,\prod_{k=1}^{d(u,\vt)} \pr\big(A_{uk} \in \dd a_{uk} \mid A_{um} = a_{um} \,\,\forall\,m \in [k-1],
A_u = a_u\big)\right] \pr\big(A_{u(d(u,\vt)+1)} > 1 \mid \mathcal{E}_{u d(u,\vt)}\big)\\
&= \left[\,\prod_{k=1}^{d(u,\vt)} 
\pr\big(\sigma_{uk} \in \log (\dd a_{uk}) \mid \mathcal{E}_{u(k-1)}\big)\right]
\pr(\sigma_{u d(u,\vt)+1)} > 0 \mid \mathcal{E}_{u d(u,\vt)}),
\end{aligned}
\]
where we abbreviate $\sigma_{\omega} = \log A_{\omega}$ and $\mathcal{E}_{u(k-1)} = \{\sigma_{um} = \log a_{um} \,\,\forall \, m \in [k-1], \sigma_u = \log a_u\}$. By the definition of the {\tt sin}-tree,
\[
\begin{array}{llll}
\pr\big(\sigma_{uk} > s \mid \mathcal{E}_{u(k-1)}\big) &=& \eee^{-\Lambda_{u(k-1)}(s)},
&s \in (\sigma_{u(k-1)},0),\\[0.2cm]
\pr\big(\sigma_{u d(u,\vt)+1)} > s \mid \mathcal{E}_{u d(u,\vt)}\big) &=& \eee^{-\Lambda_{u d(u,\vt)}(s)},
&s \in (\sigma_{u d(u,\vt)},0),
\end{array}
\]
where, by \eqref{eqn:h01-hazard}--\eqref{eqn:h-gen-hazard},
\[
\begin{aligned}
\Lambda_{u(k-1)}(s) 
&= \int_0^{s-\sigma_{u(k-1)}} \dd x\, 
\left[
\frac{1+\delta}{1+\tfrac12\delta} \left(1-\eee^{-(\sigma_{u(k-1)}-\sigma_u+x)}\right)
+ \frac{1}{1+\tfrac12\delta} \sum_{m=1}^{k-1} \left(1-\eee^{-(\sigma_{u(k-1)}-\sigma_{um}+x)}\right)
\right]\\
&= \int_{\sigma_{u(k-1)}}^s  \dd y\, \lambda_{u(k-1)}(y),
\end{aligned}
\]
with
\[
\lambda_{u(k-1)}(y) = 
\frac{1+\delta}{1+\tfrac12\delta} \left(1-\eee^{-y+\sigma_u}\right)
+ \frac{1}{1+\tfrac12\delta} \sum_{m=1}^{k-1} \left(1-\eee^{-y+\sigma_{um}}\right),
\]
and a similar expression for $\lambda_{u d(u,\vt)}(s)$. Therefore,
\[
\begin{array}{lll}
\pr\big(\sigma_{uk} \in \log (\dd a_{uk}) \mid \mathcal{E}_{u(k-1)}\big)
&=& \lambda_{u(k-1)}(\log a_{uk})\,\eee^{-\Lambda_{u(k-1)}(\log a_{uk})}\,\log(\dd a_{uk}),\\[0.2cm]
\pr(\sigma_{u(d(u,\vt)+1)} > 0 \mid \mathcal{E}_{u d(u,\vt)})
&=& \eee^{-\Lambda_{u d(u,\vt)}(0)}.
\end{array}
\]
Collecting terms, we arrive at
\begin{equation}
\label{eqn:contrepr}
\begin{aligned}
\pr(\bar\bbT^r_\emptyset = \bar\bt) 
&=  \prod_{u \in B_{r-1}(\emptyset)} \left[\prod_{k=1}^{d(u,\vt)} \frac{1}{a_{uk}}\,
\lambda_{u(k-1)}(\log a_{uk})\,\eee^{-\Lambda_{u(k-1)}(\log a_{uk})}\right]\,\eee^{-\Lambda_{u d(u,\vt)}(0)}\\
&=  \prod_{u \in B_{r-1}(\emptyset)} \prod_{k=1}^{d(u,\vt)} \frac{1}{a_{uk}}\,
\lambda_{u(k-1)}(\log a_{uk})\\
&\qquad \times \exp\left[- \sum_{\ell=0}^{r-1} \sum_{u \in \partial B_\ell(\emptyset)} \left(\sum_{k=1}^{d(u,\vt)} 
\Lambda_{u(k-1)}(\log a_{uk}) + \Lambda_{u d(u,\vt)}(0)\right) \right].
\end{aligned}\end{equation}

\medskip\noindent
{\bf 6.} 
Finally, we show that the expression in \eqref{eqn:contrepr} is the same as the one in \eqref{eqn:discrrepr}. Indeed,
with $a_{u0} = a_u$,
\[
\begin{aligned}
\lambda_{u(k-1)}(\log a_{uk}) &= \frac{1}{1+\tfrac12\delta}
\left[(1+\delta) \left(1-\frac{a_u}{a_{uk}}\right) + \sum_{m=1}^{k-1} \left(1-\frac{a_{um}}{a_{uk}}\right)\right]\\
&= \frac{1}{(1+\tfrac12\delta)\,a_{uk}} \left[\delta (a_{uk}-a_u) + \sum_{m=0}^k (a_{uk}-a_{um})\right],
\end{aligned}
\]
which shows that the two product terms match. Moreover,
\[
\begin{aligned}
\Lambda_{u(k-1)}(\log a_{uk}) 
&= \frac{1}{1+\tfrac12\delta} \Bigg[(1+\delta)\left\{\log\left(\frac{a_{uk}}{a_{u(k-1)}}\right) 
\ch{-} a_u\left(\frac{1}{a_{u(k-1)}} - \frac{1}{a_{uk}}\right)\right\}\\
&\qquad\qquad \qquad + \sum_{m=1}^{k-1} \left\{\log\left(\frac{a_{uk}}{a_{u(k-1)}}\right) 
\ch{-} a_{um} \left(\frac{1}{a_{u(k-1)}} - \frac{1}{a_{uk}}\right)\right\}\Bigg],
\end{aligned}
\]
and a similar expression for $\Lambda_{u d(u,\vt)}(0)$. Hence, by telescoping,
\[
\begin{aligned}
&\sum_{k=1}^{d(u,\vt)} \Lambda_{u(k-1)}(\log a_{uk}) + \Lambda_{u d(u,\vt)}(0)\\
&= \frac{1}{1+\tfrac12\delta} \Bigg[\delta\left\{\log\left(\frac{a_{u d(u,\vt)}}{a_u}\right)
-a_u\left(\frac{1}{a_u} - \frac{1}{a_{u d(u,\vt)}}\right)
+ \log\left(\frac{1}{a_{u d(u,\vt)}}\right) - a_{u}\left(\frac{1}{a_{u d(u,\vt)}} - 1\right)\right\}\\
&\qquad \qquad \qquad + \sum_{k=1}^{d(u,\vt)} \sum_{m=0}^{k-1} \left\{\log\left(\frac{a_{uk}}{a_{u(k-1)}}\right) 
+ a_{um} \left(\frac{1}{a_{u(k-1)}} - \frac{1}{a_{uk}}\right)\right\}\\
&\qquad \qquad \qquad + \sum_{m=0}^{d(u,\vt)} \left\{\log\left(\frac{1}{a_{u d(u,\vt)}}\right) 
- a_{um}\left(\frac{1}{a_{u d(u,\vt)}}-1\right)\right\} 
\Bigg]\\
&= \frac{1}{1+\tfrac12\delta} \Bigg[\delta \big[(a_u-1) - \log a_u\big]
+ \sum_{m=0}^{d(u,\vt)} \big[(a_{um}-1) - \log a_{um}\big]\Bigg], 
\end{aligned}
\]
which shows that the two terms in the exponents in \eqref{eqn:discrrepr} and \eqref{eqn:contrepr} match as well.
\end{proof}
 

\subsection{Tail exponents: Proof of Theorem \ref{thm:tail}}
\label{sec:tail-exp}

Recall Proposition \ref{prop:edge-equiv-bp}, which characterizes the limit degree $D_\Sr$ in terms of the total size of the edge branching process at a random time $T_1$, i.e., $|\BP_{\ve}(T_1)|$. We will use this characterization to quantify the behavior of the tail exponents.

We start with the lower bound. The parameter $\lambda$ in~\eqref{eqn:malthus-edge} is the so-called Malthusian rate of growth of the edge branching process $\BP^\circ_{\ve}$, which under the ``$\hat\mvzeta\log^+(\hat\mvzeta)$'' condition for continuous-time branching processes implies $\ee^{-\lambda t}|\BP^{\circ}_{\ve}(t)| \stackrel{a.s., L^1}{\longrightarrow} W^{\circ}_{\ve}$ (see~\cite{jagers-nerman-1,jagers-nerman-2,jagers-ctbp-book}, where the limiting random variable is finite and strictly positive). In particular, $\sup_{t \ge 0}\ee^{-\lambda t}\E[|\BP^{\circ}_{\ve}(t)|]< \infty$. For $i \ge 0$, let $W^{\circ}_i$ denote the almost sure limit of $\eee^{-\lambda t}|\BP^{\circ}_{\ve,i}(t)|$ as $t \to \infty$. Recall that the point process $\chi$ has rate $r_{\ve}^{\sss(1)}(\cdot) $ as in \eqref{eqn:rdel-one}, which satisfies $r_{\ve}^{\sss(1)}(t) = (1+\delta) r_{\ve}(t)$. Hence, using the definition of $\lambda$, we get
\begin{equation*}
\E\left(\sum_{i=0}^{\infty} \ee^{-\gl \theta_i}\right) = (1+\delta) \int_0^\infty \ee^{-\lambda t}r_{\ve}(t)\, \dd t 
= (1+\delta) \int_0^\infty \ee^{-\lambda t} \mu_{\mvzeta_{\ve}}(\dd t) = (1+\delta) < \infty.
\end{equation*}

Writing $W_{\ve} := \sum_{i=0}^{\infty} \ee^{-\lambda \theta_i}W^{\circ}_i$, we conclude from the above observations that $\ee^{-\lambda t}|\BP_{\Sr, \ve}(t)| \stackrel{a.s.}{\rightarrow} W_{\ve}$ as $t \to \infty$. Moreover, $\E[W_{\ve}] < \infty$ gives, in particular, the (almost sure) finiteness of $W_{\ve}$. Further,
\begin{align}
\label{bound-difference-to-martingale}
\E\left(\Big| \ee^{-\gl t}|\BP_{\Sr, \ve}(t)| - W_{\ve}\Big|\right) &\le \E\left(\sum_{i=0}^\infty \ee^{-\gl \theta_i} \Big| 
\ee^{-\gl (t- \theta_i)}|\BP^{\circ}_{\ve,i}(t - \theta_i)| - W^{\circ}_i\Big|\right)\\
&\qquad + \E\left(\sum_{i=\chi(t)}^\infty \ee^{-\gl \theta_i} W^{\circ}_i\right).
\end{align}
The convergence to zero of the right-hand side of \eqref{bound-difference-to-martingale} follows from the dominated convergence theorem, and the independence of $\chi$ and $(\BP^{\circ}_{\ve,i})_{i \in \mathbb{N}_0}$. Thus, there exist $\delta, t_0, \eta >0$, such that
\begin{equation}
\label{eqn:248}
 \pr(\eee^{-\lambda t}|\BP_{\ve}(t)| > \delta ) \geq \eta \qquad \forall~t\geq t_0.
\end{equation}
Hence, for $k \geq \delta \eee^{\lambda t_0}$ and $C =\eta \delta^{1/\lambda}$,
\[
\pr(D_{\Sr} \geq k) = \int_0^\infty \eee^{-t} \pr(|\BP_{\ve}(t)|\geq k)\, \dd t 
\geq \int_{\frac{\log (k/\delta)}{\lambda}}^{\infty} \eee^{-t} \pr(|\BP_{\ve}(t)|\geq k)\, \dd t \geq \frac{C}{k^{1/\lambda}}.   
\]
To prove the upper bound, once again the moment bound in Proposition~\ref{prop:zeta-bounds} and the Markov inequality give that, for any $x>0$ and $t\geq  0$,
\[
\pr(|\BP_{\ve}(t)| > x) \leq \frac{H_l}{[\eee^{-\lambda t} x]^l}.
\]
Thus,
\begin{align*}
\pr(D_{\Sr} \geq k) &=\int_0^\infty \eee^{-t} \pr(|\BP_{\ve}(t)|\geq k)\,\dd t 
\leq \int_0^{\frac{\log{k}}{\lambda}} \frac{H_l}{[\eee^{-\lambda t} x]^l} \eee^{-t}\, \dd t 
 + \int_{\frac{\log{k}}{\lambda}}^\infty \eee^{-t}\, \dd t \\
 &\le \left[\frac{H_l}{\lambda l -1} + 1\right] \frac{1}{k^{1/\lambda}},
\end{align*}
which settles the claim. 


\subsection{Root degree behavior: Proof of Theorem~\ref{degree_growth}} 
\label{sec:max-deg-proof}

We will focus on the growth of the root degree, i.e., either $d(0,\cT_n)$ or $d(1,\cT_n)$. The growth of any fixed vertex can be treated similarly. 

Consider the evolution of the degree of the root using the weight of the vertex $\theta(v_1,n)$ as in \eqref{eq:weightchoice}. Suppose that at time $n$, the degree of vertex $v_1$ is $m$, and suppose that these additions happened at times $t_1\equiv=1, t_2, \dots, t_m$ (see Figure \ref{fig:times}). 

\begin{figure}[htbp]
	\centering
		\begin{tikzpicture}[x=1.8cm,y=1cm, line cap=round, line join=round]

		\draw[thick] (0,0) -- (2.9,0);
		\draw[thick] (2.9,0) -- (5.5,0);
		\draw[thick] (6.8,0) -- (8.1,0);

		\draw[thick,dotted] (5.5,0) -- (6.8,0);

		\draw[thick] (0,-0.28) -- (0,0.28);
		\draw[thick] (2,-0.38) -- (2,0.38);
		\draw[thick] (4.3,-0.38) -- (4.3,0.38);
		\draw[thick] (6.8,-0.38) -- (6.8,0.38);
		\draw[thick] (8.1,-0.30) -- (8.1,0.30);

		\node[below] at (0,-0.45) {$t_1=1$};
		\node[below] at (2,-0.45) {$t_2$};
		\node[below] at (4.3,-0.45) {$t_3$};
		\node[below] at (6.8,-0.45) {$t_m$};
		\node[below] at (8.1,-0.45) {$n$};

		\node[below] at (1.0,-0.15) {$1$};
		\node[below] at (3.15,-0.15) {$2$};
		\node[below] at (7.45,-0.15) {$m$};

		\end{tikzpicture}
	\caption{Times of growth of degree of vertex $v_1$.}
	\label{fig:times}
\end{figure}

Note that the weight used for the probability of connection when the network transitions from $n \leadsto n+1$ can be rewritten as 
\eqan{
\frac{\theta(v_1, n)}{(1+\frac{1}{2}\delta)n(n+1)} 
&= \frac{(n-v+1)n-\sum_{i=1}^n 
\sum_{j=v}^n \indic{s<t_i}}{(1+\frac{1}{2}\delta)n(n+1)}=\frac{n(m+\delta) 
- (t_1+t_2+\cdots+t_m)}{(1+\frac{1}{2}\delta)n(n+1)}.
}
Hence, for keeping track of the evolution of this vertex, it is easier to move to continuous time and consider the following point process driven by a Yule process:

\begin{defn}[Degree evolution in continuous time]
	Let $(Y(t))_{t \ge 0}$ be a standard rate-$1$ Yule process started with $Y(0) = 2$ individuals. Suppose that new individuals are marked according to the following rule. Let $D(t)$ be the number of marked individuals at time $t$, with $D(0) = 1$. Let $T_1, T_2, \dots, T_{D(t)}$ be the times when these marks occurred. Define $W(t) = \sum_{j=1}^{D(t)} Y(T_j)$. When a new individual is born at time $t$ (which occurs at rate $Y(t)$), then it is marked with probability
	\begin{equation*}
	    p_\delta(t) = \gamma \left( \frac{D(t) + \delta}{Y(t)+1} - \frac{W(t)}{Y(t)(Y(t)+1)} \right),
	\end{equation*}
where $\delta>-1$ is a constant, and $\gamma = \frac{1}{1 + \delta/2} = \frac{2}{2+\delta}$.
\end{defn}

Abbreviate $\tau_n = \inf\set{t \geq 0\colon\, Y(t) = n}$. It is easy to check that $(d(v_{1},n))_{n\geq 1} \stackrel{d}{=}  (D(\tau_n))_{n\geq 1}$. To simplify the algebra, we use $Y(t)$ instead of $Y(t)+1$ in the denominator, i.e.,
\begin{equation}
	    p_\delta(t) = \gamma \left( \frac{D(t) + \delta}{Y(t)} - \frac{W(t)}{Y(t)^2} \right).
\end{equation}
This does not change the asymptotics, but makes the computations easier to parse. The following proposition, together with the fact that for Yule process dynamics $\tau_n - \log{n} \convas W$ for a finite random variable $W$, completes the proof of the theorem:

\begin{prop}[Limit of $D(t)$]
For any $\delta >-1$, there exists a non-negative random variable $X_\delta$ with $\E[X_\delta] >0$ such that, almost surely and in $L^1$, as $t \to \infty$,
\begin{equation}
    \eee^{-t/\phi(\delta)} D(t) \longrightarrow X_\delta,
\end{equation}
with $\phi(\delta)$ as in \eqref{phidef}, so that
\begin{equation}
    \frac{1}{\phi(\delta)} = \frac{1}{2} \left( \sqrt{1 + 4\gamma} - 1 \right). 
\end{equation}
\end{prop}

\begin{proof}
We start by considering the evolution of the various processes in continuous time. The semi-martingale decomposition of the associated processes and their analysis will form the heart of the proof. Define the state vector $Z(t) = (Z_1(t), Z_2(t))^\top$, where 
\[
Z_1(t) = D(t), \qquad  Z_2(t) = \frac{W(t)}{Y(t)}.
\]
Due to the Yule process dynamics, jumps occur at rate $Y(t)$. A mark occurs with rate $\lambda_1(t) = Y(t)p_\delta(t) = \gamma D(t) - \gamma Z_2(t) + \gamma \delta$, causing jumps $\Delta D(t) = 1$ and $\Delta W(t) = Y(t)+1$. No mark occurs with rate $\lambda_2(t) = Y(t) - \lambda_1(t)$, causing jumps $\Delta D(t) = 0$ and $\Delta W(t) = 0$. Applying the infinitesimal generator $\mathcal{L}$ of the Markov jump process to $Z(t)$, we get
\begin{align}
     \mathcal{L} Z_1(t)&= \mathcal{L} D(t)= \gamma D(t) - \gamma Z_2(t) + \gamma \delta, \\
    \mathcal{L} Z_2(t) &= \gamma D(t) - (1+\gamma) Z_2(t) + \gamma \delta + R_0(t),
\end{align}
where the remainder term is 
\[
R_0(t) = \frac{W(t)}{Y(t)(Y(t)+1)}.
\] 
Note that $W(t) \le D(t)Y(t) \le Y(t)^2$ since marks are assigned to a subset of the population at or before time $t$. Thus, 
\begin{equation}
\label{eqn:det-bounds}
	Z_2(t) \le D(t) \le Y(t), \qquad R_0(t) \le \frac{Y(t)^2}{Y(t)^2} \le 1.
\end{equation}

Next, we consider the semimartingale decomposition of $Z(\cdot)$. By the Doob-Meyer decomposition, $dZ(t) = \mathcal{L}Z(t)dt + dM(t)$, yielding the matrix form
\begin{equation}
	\label{eqn:semi-mart-decomp}
    \dd Z(t) = A_\delta Z(t)\, \dd t + S\, \dd t + R(t)\, \dd t + \dd M(t),
\end{equation}
where $M(t) = (M_1(t), M_2(t))$ is the martingale term, and 
\begin{equation}
\label{eqn:matrix}
	A_\delta = \begin{pmatrix} \gamma 
	& -\gamma \\ \gamma & -(1+\gamma) \end{pmatrix}, 
	\qquad S = \begin{pmatrix} \gamma\delta \\ \gamma\delta \end{pmatrix}, 
	\qquad R(t) = \begin{pmatrix} 0 \\ R_0(t) \end{pmatrix}.
\end{equation}
The following lemma collects some salient facts about the matrix $A_\delta$:

\begin{lem}[Spectral properties of the drift matrix] 
\label{lem:spectral}
Let $\delta>-1$ and $\gamma = \frac{2}{2+\delta}$. Define the drift matrix $A_\delta$ as in \eqref{eqn:matrix}. Clearly, $A_\delta$ has two distinct real eigenvalues given by
\begin{equation}
    \lambda_+ = \frac{1}{\phi(\delta)} = \frac{1}{2}\left(\sqrt{1+4\gamma}-1\right)>0, \quad 
    \lambda_-  = \frac{1}{2}\left(-\sqrt{1+4\gamma}-1\right)<0.
\end{equation}
Furthermore, the associated left-eigenvectors $\vv^\top$ for $\lambda_+$, and $\uu^\top$ for $\lambda_-$, normalized such that their first coordinates are $1$, are given by
\begin{equation}
    \vv = \begin{pmatrix} 1 \\ \frac{1/\phi(\delta) - \gamma}{\gamma} \end{pmatrix}, 
    \quad \uu = \begin{pmatrix} 1 \\ \frac{\lambda_- - \gamma}{\gamma} \end{pmatrix}.
\end{equation}
Finally, for the principal left-eigenvector $\vv$, the second coordinate is strictly negative ($v_2 < 0$), and its magnitude is strictly bounded by the first coordinate ($|v_2| < v_1$). The matrix of eigenvectors is invertible.
\end{lem}

\begin{proof}[Proof of Lemma \ref{lem:spectral}]
	The various formulas for eigen-values and eigen-vectors can be easily checked. To show that $v_2 < 0$, observe that $1/\phi(\delta) < \gamma$ is equivalent to $\sqrt{1+4\gamma} - 1 < 2\gamma$, which holds since squaring $1+2\gamma$ gives $1+4\gamma+4\gamma^2 > 1+4\gamma$. Since $v_2$ is negative, its absolute value is
\begin{equation}
    |v_2| = \frac{\gamma - \phi(\delta)}{\gamma} = 1 - \frac{\phi(\delta)}{\gamma}.
\end{equation}
Since the principal eigenvalue is strictly positive and $\gamma > 0$, the ratio satisfies $1/(\phi(\delta)\gamma)> 0$. Therefore, $|v_2| = 1 -1/(\phi(\delta)\gamma) < 1$. However, $v_1 = 1$, and hence $|v_2| < v_1$.
\end{proof}

We continue the proof by projecting $M$ in one direction, using $\vv$ to obtain a scalar process. Define $X(t) = \vv^\top Z(t)$. Using the semi-martingale decomposition \eqref{eqn:semi-mart-decomp} and Lemma \ref{lem:spectral}, we get
\begin{equation}
	\label{eqn:semi-mart-x}
    \dd X(t) = (\phi(\delta))^{-1} X(t)\, \dd t + (\vv^\top S)\, \dd t + v_2 R_0(t)\, \dd t + \dd M_X(t),
\end{equation}
where $M_X(t) = \vv^\top M(t)$. Further note that, because $Z_2(t) \leq D(t)$, 
\begin{equation}
\label{eqn:bound-x}
	X(t) = v_1 D(t) - |v_2| Z_2(t) \geq \frac{1}{\phi(\delta)\gamma}D(t). 
\end{equation}
We use these inequalities to get a priori bounds on the expectations of $X$ and hence of $D$. Indeed, by \eqref{eqn:semi-mart-x}, 
\begin{equation}
    \frac{d}{dt} \E[X(t)] = (\phi(\delta))^{-1} \E[X(t)] + \vv^\top S + v_2 \E[R_0(t)].
\end{equation}
Since $R_0(t) \le 1$, the derivative is bounded above by $(\phi(\delta))^{-1} \E[X(t)] + K$, where $K = \vv^\top S + 1$. By Gr\"onwall's inequality,
\begin{equation}
    \mathbb{E}[X(t)] \le X(0)\, \ee^{t/\phi(\delta)} + {K}{\phi(\delta)} 
    \left( \ee^{t/\phi(\delta)} - 1 \right) \le C_1 \ee^{t/\phi(\delta)}.
\end{equation}
Thus, there exists a constant $C_2$ such that, for all $t \ge 0$,
\begin{equation}
	\label{eqn:apriori}
    \mathbb{E}[D(t)] \le C_2\, \ee^{t/\phi(\delta)}.
\end{equation}

Finally, we analyse the quadratic variation of the process $\ee^{-(\phi_{\delta})^{-1}t}X(t)$ and show that this is bounded for all $t$. Therefore, the corresponding martingale term is $L^2$-bounded and converges in $L^2$. Applying It\^o's lemma, we get that
\begin{equation}
    \ee^{-t/\phi(\delta)} X(t) = X(0) + \int_0^t \ee^{-s/\phi(\delta)} (\vv^\top S + v_2 R_0(s))\, \dd s + N(t),
\end{equation}
where $N(t) = \int_0^t \ee^{-s/\phi(\delta)} \dd M_X(s)$ is the discounted local martingale. We bound its predictable quadratic variation $\langle N \rangle_t$.  Using the normalized principal left-eigenvector $v = (1, v_2)^\top$ from Lemma \ref{lem:spectral}, we see that the jumps are given by $\Delta X(t) = \Delta D(t)  + v_2 \Delta Z_2(t)$. These are governed by two events:
\begin{enumeratea}
    \item 
    \textbf{Mark Event} (Rate $\lambda_1 \le \gamma D(t) $): Jumps are $\Delta D(t)  = 1$ and $\Delta Z_2(t)  = 1 - R_0(t) $. The scalar jump is $\Delta X(t)  = 1 + v_2(1 - R_0(t))$. Because $0 \le R_0(t) \le 1$ and $v_2 \in (-1, 0)$, the quantity $v_2(1-R_0(t))$ lies in the interval $[v_2, 0]$. Therefore, $0 < 1 + v_2 \le \Delta X(t) \le 1$. Consequently, $(\Delta X(t))^2 \le 1$. The contribution to the variation rate is strictly bounded by $\gamma D(t)$. 
    \item 
    \textbf{No-Mark Event} (Rate $\lambda_2 \le Y(t) $): Jumps are $\Delta D(t)  = 0$ and $\Delta Z_2(t)  = -R_0(t)$. The scalar jump is $\Delta X(t) = -v_2 R_0(t) $. Since $R_0(t) \le D(t)/Y(t)$, we have $(\Delta X(t))^2 = v_2^2 R_0(t)^2 \le v_2^2 \left(\frac{D(t)}{Y(t)}\right)^2$. The contribution to the variation rate is bounded by $Y(t) \cdot v_2^2 \frac{D(t)^2}{Y(t)^2} = v_2^2 \frac{D(t)^2}{Y(t)}$.
\end{enumeratea}

Summing these contributions, we obtain the differential bound 
	\[
	\dd\langle M_X \rangle_t \le \Big(\gamma D(t) + v_2^2 \frac{D(t)^2}{Y(t)}\Big)\, \dd t.
	\]
Since marks are a subset of the population, we have $D(t) \le Y(t)$, which allows us to deterministically bound the non-linear term as 
\[
\frac{D(t)^2}{Y(t)} \le D(t) \frac{Y(t)}{Y(t)} = D(t).
\]
Therefore, taking $K_2 = \gamma + v_2^2$, we have a strictly linear bound given by
\begin{equation}
    \dd\langle M_X \rangle_t \le K_2 D(t)\, \dd t.
\end{equation} 
Using the previous a priori bounds on the expectation of $D(\cdot)$, we can bound the expectation of the total predictable quadratic variation of the discounted martingale:
\begin{equation}
    \mathbb{E}[\langle N \rangle_\infty] = \mathbb{E}\left[\int_0^\infty \ee^{-2t/\phi(\delta)} 
    \dd\langle M_X \rangle_t \right] \le K_2 \int_0^\infty \ee^{-2t/\phi(\delta)} \mathbb{E}[D(t)]\, \dd t.
\end{equation}
Substituting our a priori bound \eqref{eqn:apriori}, we get
\begin{equation}
    \mathbb{E}[\langle N \rangle_\infty] \le K_2 C_2 \int_0^\infty \ee^{-t/\phi(\delta)} \dd t 
    = \frac{K_2 C_2}{\phi(\delta)} < \infty.
\end{equation}
Because $\phi(\delta) > 0$ for all $\delta >-1$, this integral is strictly finite. Hence, the martingale $N(t)$ is bounded in $L^2$, making it uniformly integrable and guaranteeing almost sure and $L^1$ convergence of $N(t)$ to a finite limit $N_\infty$. 

Next, we study the continuous drift integral 
\[
\int_0^t e^{-s/\phi(\delta)} (v^\top S + v_2 R_0(s)) ds.
\] 
While $R_0(s)$ is random, it satisfies $0 \le R_0(s) \le 1$ almost surely for all $s \ge 0$. Thus, for almost every sample path, the random integrand above is bounded in absolute value by the deterministic function $e^{-s/\phi(\delta)} (|v^\top S| + |v_2|)$. Since $\phi(\delta) \in (0,\infty)$, this bounding function is strictly integrable on $[0, \infty)$. Therefore, the random drift integral converges absolutely and almost surely to a finite limit as $t \to \infty$. By DCT this integral also converges in $L^1$. Since both the drift integral and the martingale $N(t)$ converge almost surely and in $L^1$, the discounted projection $\ee^{-t/\phi(\delta)} X(t)$ converges almost surely and in $L^1$ to a finite random variable $X_\infty$.

To conclude the proof, we must establish that the normalized counting process $\ee^{-\phi(\delta) t} D(t)$ converges to a limit $D_\infty$, and that this limit is positive with positive probability.

First, we prove convergence. By Lemma \ref{lem:spectral}, the drift matrix $A_\delta$ possesses a secondary left-eigenvector $u$ associated with the strictly negative eigenvalue $\lambda_-$. Defining the secondary projection $U(t) = \vu^\top Z(t)$, the same semimartingale machinery guarantees that $\mathbb{E}[U(t)]$ is bounded, and so
\[
\ee^{-t/\phi(\delta)} U(t) \xrightarrow{L^1} 0.
\] 
Since the eigenvectors $\vv$ and $\vu$ are linearly independent, $D(t)$ can be expressed as a fixed linear combination $D(t) = c_1 X(t) + c_2 U(t)$. Thus, we now get $\ee^{-t/\phi(\delta)} D(t)$ converges almost surely and in $L^1$ to a finite limit $D_\infty = c_1 X_\infty$.

Second, we prove that the expectation of this limit is strictly positive. By definition, $X(t) = D(t) - |v_2| Z_2(t)$. Because $Z_2(t) \ge 0$, we have the deterministic upper bound $X(t) \le D(t)$ for all $t \ge 0$. Multiplying by the discount factor and taking the limit $t \to \infty$ preserves this inequality, so that
\begin{equation}
    X_\infty \le D_\infty \quad \text{a.s.}
\end{equation}

To complete the proof it is enough to show that $\E[X_\infty]>0$. We analyze the exact transition dynamics of the discrete embedded chain $Z_n = (D_n, Z_{2,n})^\top$ evaluated exactly when the Yule population is $Y(t) = n$. Let $X_n = \vv^\top Z_n$ denote the corresponding discrete time version $X(\cdot)$. The following lemma completes the proof of Theorem~\ref{degree_growth}:

\begin{lem}[Lower bound on expectation]
\label{lem:lowerbound}
There exists $c_2>0$ such that $\E[X_n] \geq c_2 n^{1/\phi(\delta)}$.
\end{lem}

\begin{proof}
The probability of a mark at transition $n \to n+1$ is $p_n = \frac{\gamma D_n}{n} - \frac{\gamma Z_{2,n}}{n} + \frac{\gamma \delta}{n}$. The conditional expectations for the components of $Z_{n+1}$ are
\begin{align}
    \mathbb{E}[D_{n+1} \mid \mathcal{F}_n] 
    &= D_n + p_n, \\
    \mathbb{E}[Z_{2,n+1} \mid \mathcal{F}_n] 
    &= \mathbb{E}\left[ \frac{W_n + (n+1)I_{n+1}}{n+1} \mathrel{\Big|} \mathcal{F}_n \right] 
    = \frac{W_n}{n+1} + p_n = Z_{2,n} - \frac{Z_{2,n}}{n+1} + p_n.
\end{align}
We extract the exact $1/n$ drift to match the continuous generator $A_\delta$. Using the identity $-\frac{1}{n+1} = -\frac{1}{n} + \frac{1}{n(n+1)}$, we rewrite the second expectation as
\begin{equation}
    \mathbb{E}[Z_{2,n+1} \mid \mathcal{F}_n] = Z_{2,n} + p_n - \frac{Z_{2,n}}{n} + \frac{Z_{2,n}}{n(n+1)}.
\end{equation}
Writing this system in matrix form, utilizing $A_\delta$ and $S$, we get an exact algebraic expansion with no asymptotic error terms, given by
\begin{equation}
    \mathbb{E}[Z_{n+1} \mid \mathcal{F}_n] = Z_n + \frac{1}{n} A_\delta Z_n + \frac{1}{n} S 
    + \begin{pmatrix} 0 \\ \frac{Z_{2,n}}{n(n+1)} \end{pmatrix}.
\end{equation}
Projecting this system onto the principal left-eigenvector $v^\top$, we get the exact discrete drift for the scalar process $X_n = \vv^\top Z_n$, i.e.,
\begin{equation}
    \mathbb{E}[X_{n+1} \mid \mathcal{F}_n] = X_n \left( 1 + \frac{1}{\phi(\delta)n} \right) 
    + \frac{v^\top S}{n} + \frac{v_2 Z_{2,n}}{n(n+1)}.
\end{equation}
It remains to bound the remainder term $\frac{v_2 Z_{2,n}}{n(n+1)}$. Because $v_2 < 0$ and $Z_{2,n} \ge 0$, this term is negative. From the bound $Z_{2,n} \le D_n$, we established that $X_n \ge (1-|v_2|)D_n$, which implies $D_n \le \frac{X_n}{1-|v_2|}$. Thus,
\begin{equation}
    v_2 Z_{2,n} = -|v_2| Z_{2,n} \ge -|v_2| D_n \ge -\left(\frac{|v_2|}{1-|v_2|}\right) X_n.
\end{equation}
Let $C = \frac{|v_2|}{1-|v_2|} > 0$. Taking the unconditional expectation of the $X_n$ recurrence, and noting the source term $\vv^\top S \ge 0$, we obtain the strict difference inequality
\begin{equation}
    \mathbb{E}[X_{n+1}] \ge \mathbb{E}[X_n] \left( 1 + \frac{1}{\phi(\delta)n} - \frac{C}{n(n+1)} \right).
\end{equation}
For all sufficiently large $n$, the multiplier is strictly positive. Iterating this inequality from a starting index $N$ yields:
\begin{equation}
    \mathbb{E}[X_n] \ge \mathbb{E}[X_N] \prod_{k=N}^{n-1} \left( 1 + \frac{1}{\phi(\delta)k} - \frac{C}{k(k+1)} \right).
\end{equation}
Because the perturbation $\frac{C}{k(k+1)}$ is absolutely summable, standard analysis of infinite products guarantees that this product diverges at exactly the same polynomial rate as $\prod (1 + \frac{1}{\phi(\delta)k})$. Therefore there exists a constant $c_0 > 0$ such that, for all $n$,
\begin{equation}
    \mathbb{E}[X_n] \ge c_0 n^{1/\phi(\delta)}.
\end{equation}
This completes the proof of Lemma \ref{lem:lowerbound}.
\end{proof}
Lemma \ref{lem:lowerbound} completes the proof of Theorem~\ref{degree_growth}. 
\end{proof}


\section*{Acknowledgments}
SB was supported by the National Science Foundation (NSF) through grants DMS-2113662, DMS-2413928, DMS-2434559, and through NSF RTG grant DMS-2134107. FdH and RvdH were supported by the Netherlands Organisation for Scientific Research (NWO) through Gravitation-grant NETWORKS-024.002.003. RR was supported by the Office of Naval Research under the Vannevar Bush Faculty Fellowship N0014-21-1-2887. SB, FdH, RvdH and RR were supported by the NSF under grant DMS-1928930 while in residence at the Simons Laufer Mathematical Sciences Institute, Berkeley, California, USA during the Spring 2025 semester.

 
\bibliographystyle{chicago}

\bibliography{bibliofile.bib}


\end{document}

%% file: style.tex
\usepackage{chicago}
\usepackage{systeme}
\usepackage{mathrsfs,xfrac} 
\usepackage[colorlinks=true,linkcolor=blue,citecolor=blue,urlcolor=blue]{hyperref} 
\usepackage{amsmath,amssymb,amsthm,amsfonts,amsbsy,latexsym,dsfont,color} 
\usepackage[numeric,initials,nobysame]{amsrefs} 
\usepackage[foot]{amsaddr}

\usepackage[utf8]{inputenc}
\usepackage[english]{babel}
\usepackage{comment}

\usepackage{enumerate}

\newenvironment{enumeratea}{\begin{enumerate}[\upshape a)]\setlength{\itemsep}{1ex}}{\end{enumerate}}

\usepackage{paralist}

\numberwithin{equation}{section}

\newcommand{\eqn}[1]{\begin{equation}#1\end{equation}}
\newcommand{\eqan}[1]{\begin{align}#1\end{align}}

\newtheorem{thm}{Theorem}[section]
\newtheorem{lem}[thm]{Lemma}
\newtheorem{cor}[thm]{Corollary}
\newtheorem{prop}[thm]{Proposition}
\newtheorem{defn}[thm]{Definition}

\theoremstyle{definition}

\renewcommand{\le}{\leqslant} 
\renewcommand{\ge}{\geqslant} 
\renewcommand{\leq}{\leqslant} 
\renewcommand{\geq}{\geqslant}

\newcommand{\ind}{\mathds{1}}

\newcommand{\abs}[1]{\left\vert#1\right\vert}
\newcommand{\set}[1]{\left\{#1\right\}}

\newcommand{\equald}{\stackrel{\mathrm{d}}{=}}
\newcommand{\probc}{\stackrel{\sss\prob}{\longrightarrow}}

\usepackage{bbm}
\newcommand{\indic}[1]{\mathbbm{1}_{\{#1\}}}

     \let\gl=\lambda

\newcommand{\cE}{\mathcal{E}}

\newcommand{\cL}{\mathcal{L}}
\newcommand{\cM}{\mathcal{M}}
\newcommand{\cP}{\mathcal{P}}
\newcommand{\cT}{\mathcal{T}}

\newcommand{\vA}{\mathbf{A}}

\newcommand{\vQ}{\mathbf{Q}}

\newcommand{\ve}{\mathbf{e}}

\newcommand{\vs}{\mathbf{s}}\newcommand{\vt}{\mathbf{t}}\newcommand{\vu}{\mathbf{u}}
\newcommand{\vv}{\mathbf{v}}
 
\newcommand{\mvzeta}{\boldsymbol{\zeta}}

\newcommand{\fP}{\mathfrak{P}}

\newcommand{\fm}{\mathfrak{m}}

\newcommand{\bN}{\mathbb{N}}
\newcommand{\bR}{\mathbb{R}}
\newcommand{\bT}{\mathbb{T}}

\newcommand{\sE}{\mathscr{E}}

\DeclareMathOperator{\E}{\mathbb{E}}
\DeclareMathOperator{\pr}{\mathbb{P}}

\newcommand{\FF}{\mathcal{F}}
\newcommand{\bbT}{\mathbb{T}}
\newcommand{\TT}{\mathcal{T}}
\newcommand{\bs}{\mathbf{s}}

\newcommand{\bt}{\mathbf{t}}
\newcommand{\bfomega}{{\boldsymbol \omega}}
\newcommand{\Zbold}{{\mathbb{Z}}}
\newcommand{\prob}{\mathbb{P}}
\definecolor{aqua}{rgb}{0.0, 1.0, 1.0}
\definecolor{boo}{rgb}{1.0, 0.0, 1.0}

\newcommand{\probfr}{\stackrel{{\sss\mbox{$\prob$-\bf fr}}}{\longrightarrow}}
\newcommand{\probcrf}{\stackrel{{\sss\mbox{$\prob$-\bf efr}}}{\longrightarrow}}

\newcommand{\Efr}{\stackrel{{\sss\mbox{$\E$-\bf fr}}}{\longrightarrow}}

\newcommand{\convas}{\stackrel{\mathrm{a.s.}}{\longrightarrow}}
\newcommand{\convd}{\stackrel{d}{\longrightarrow}}
\newcommand{\convp}{{\stackrel{\sss\mathbb{P}}{\longrightarrow}}}

\newtheorem{lemma}[thm]{Lemma}

\newcommand{\stod}{\preceq_{\mathrm{st}}}
\DeclareMathOperator{\BP}{BP}

\newcommand{\age}{{\sf Age} }

\newcommand{\Sr}{{\sf SeRi} }

\newcommand{\fringe}{{\sf fringe} }

\usepackage[mathscr]{euscript}
\DeclareMathAlphabet{\mathscrbf}{OMS}{mdugm}{b}{n}

\DeclareFontFamily{U}{BOONDOX-calo}{\skewchar\font=45 }
\DeclareFontShape{U}{BOONDOX-calo}{m}{n}{
  <-> s*[1.05] BOONDOX-r-calo}{}
\DeclareFontShape{U}{BOONDOX-calo}{b}{n}{
  <-> s*[1.05] BOONDOX-b-calo}{}
\DeclareMathAlphabet{\mathcalb}{U}{BOONDOX-calo}{m}{n}
\SetMathAlphabet{\mathcalb}{bold}{U}{BOONDOX-calo}{b}{n}
\DeclareMathAlphabet{\mathbcalb}{U}{BOONDOX-calo}{b}{n}

\newcommand{\sss}{\scriptscriptstyle}

\usepackage{tikz}
\usetikzlibrary{calc,intersections,through,backgrounds,shapes.geometric}
\usetikzlibrary{graphs}

\newcommand{\Cod}{\mathcalb{CoD}}

\usetikzlibrary{calc,intersections,through,backgrounds,shapes.geometric}
\usetikzlibrary{graphs}
\tikzset{every path/.style={line width=.07 cm}}

\usepackage{subcaption}

\newcommand{\dd}{\mathrm{d}}
\newcommand{\eee}{\mathrm{e}}

\newcommand{\uu}{\mathbf{u}}

%% file: discrete_tree.tex

\begin{figure}[htbp]
\centering

\begin{tikzpicture}[scale=0.9, every node/.style={font=\normalsize}]
    \node[circle, draw, thick] (root) at (0,0) {$~u~$};
    \node[circle, draw, thick] (u1) at (-2,-1.8) {$u1$};
    \node[circle, draw, thick] (u2) at (-0.9,-1.8) {$u2$};
    \node[circle, draw, thick] (udu) at (1.8,-1.8) {$ud$};
    \foreach \x in { 0.2,0.5,0.8} {
            \draw[fill=black] (\x-0.1,-1.8) circle(1pt);
        }

    \draw[thick] (root) -- (u1);
    \draw[thick] (root) -- (u2);
    \draw[thick] (root) -- (udu);

    \begin{scope}[xshift=7cm, yshift=-1cm, scale=1.3] 
        \draw[thick] (-1.5,0)--(5,0); 

        \foreach \x/\label in {-1.5/1,0/\tau_u, 1.25/\tau_{u1}, 2.5/\tau_{u2}, 4.25/\tau_{u d(u,\vt)}, 5/n} {
            \draw[fill=black] (\x,0) circle(1pt);
            \node[below=3pt, font=\normalsize] at (\x,0) {$\label$};
        }
        
        \foreach \x in { 3,3.25, 3.5} {
            \draw[fill=black] (\x+0.25,0) circle(1pt);
        }
        \node[below=22pt, font=\normalsize] at (0,0) {$\tau_{u0}$};
        \node[below=20pt, rotate=90, font=\large] at (-0.15,0) {$=$};

        \node[below=20pt, rotate=90, font=\large] at (4.87,0) {$=$};
        \node[below=22pt, font=\normalsize] at (5,0) {$\tau_{u(d(u,\vt)+1)}$};

        \node[below=3pt, font=\normalsize] at (-1.5,0) {$1$};
    \end{scope}
\end{tikzpicture}
\caption{\small\textit{Left:} $B_{1}(u)$ in the Ulam-Harris notation, with offspring $u1, u2, \dots, u{d(u,\vt)}$ and $d=d(u,\vt)$. \\
\textit{Right:} A visual representation of the vertices \(\bbT^r_{o_n} \) in \([n]\) corresponding to the vertices in \(B_{1}(u)\), with the convention \(\tau_{u0}=\tau_{u}\) and \(\tau_{u(d(u,\vt)+1)}=n\).}
\label{fig:discrete_tree}
\end{figure}